\documentclass[11pt,reqno]{amsart}
\usepackage{amssymb,mathrsfs,graphicx,subfigure, enumerate}
\usepackage{amsmath,amsfonts,amssymb,amscd,amsthm,bbm}
\usepackage{extpfeil}
\usepackage{graphicx,colortbl}
\usepackage{float}
\usepackage{epsfig}
\usepackage{caption}
\usepackage{tikz}
\usepackage{tikz-cd}
\usepackage{bm}
\graphicspath{{Figure/}}\usepackage[margin=1in]{geometry}

\makeatletter
\@namedef{subjclassname@2020}{\textup{2020} Mathematics Subject Classification}
\makeatother

\title[Hydrodynamic limits of CSH system]
{Quantitative Hydrodynamic Limit of the Chern--Simons--Higgs System}

\author[Kim]{Jeongho Kim}
\address[Jeongho Kim]{\newline Department of Applied Mathematics, Kyung Hee University\newline 1732 Deogyeong-daero, Giheung-gu, Yongin-si, Gyeonggi-do 17104, Republic of Korea}
\email{jeonghokim@khu.ac.kr}

\author[Moon]{Bora Moon}
\address[Bora Moon]{\newline Department of Mathematics, Yonsei University, Seoul 03722, Republic of Korea}
\email{boramoon@yonsei.ac.kr}

\begin{document}
	\newtheorem{theorem}{Theorem}[section]
	\newtheorem{lemma}{Lemma}[section]
	\newtheorem{corollary}{Corollary}[section]
	\newtheorem{proposition}{Proposition}[section]

	\newtheorem{definition}{Definition}[section]
	
	\theoremstyle{remark}
	\newtheorem{remark}{Remark}[section]

	\renewcommand{\theequation}{\thesection.\arabic{equation}}
	\renewcommand{\thetheorem}{\thesection.\arabic{theorem}}
	\renewcommand{\thelemma}{\thesection.\arabic{lemma}}
	\renewcommand{\theremark}{\thesection.\arabic{remark}}
	
	\newcommand{\bbr}{\mathbb R}
	\newcommand{\bbz}{\mathbb Z}
	\newcommand{\bbn}{\mathbb N}
	\newcommand{\bbs}{\mathbb S}
	\newcommand{\bbp}{\mathbb P}
	\newcommand{\bbt}{\mathbb T}
	\newcommand{\ddiv}{\textrm{div}}
	\newcommand{\bn}{\bf n}
	\newcommand{\rr}[1]{\rho_{{#1}}}
	\newcommand{\thh}{\theta}
	\def\charf {\mbox{{\text 1}\kern-.24em {\text l}}}
	\renewcommand{\arraystretch}{1.5}
	
	\newcommand{\T}{\mathbb{T}}
	\newcommand{\N}{\mathbb{N}}
	\newcommand{\R}{\mathbb{R}}
	\newcommand{\lt}{\left}
	\newcommand{\rt}{\right}
	\newcommand{\bq}{\begin{equation}}
	\newcommand{\eq}{\end{equation}}
	\newcommand{\e}{\varepsilon}
	\newcommand{\mc}{\mathcal{C}}
	\newcommand{\pa}{\partial}
	\newcommand{\ph}{\hat{p}}
	\renewcommand{\d}{\textup{d}}
	\renewcommand{\i}{\textup{i}}
	\renewcommand{\Re}{\textup{Re}} 
	\renewcommand{\Im}{\textup{Im}}
	\newcommand{\p}{\partial}

	\newcommand{\circlednum}[1]{%
		\tikz[baseline=(char.base)]{
			\node[shape=circle,draw,inner sep=0.6pt,scale=0.8] (char) {#1};}
	}
	
	
	
	\thanks{J. Kim was supported by Samsung Science and Technology Foundation under Project Number SSTF-BA2401-01. The work of B. Moon was supported by Basic Science Research Program through the National Research Foundation of Korea (NRF) funded by the Ministry of Education (RS-2022-NR074808) and funded by the Korea government(MSIT) (RS-2024-00406821).}

\begin{abstract}
	We study the hydrodynamic limit of the Chern--Simons--Higgs system,
	a relativistic gauge field model involving the Chern--Simons interaction.
	We introduce a single scaling parameter capturing both the non-relativistic
	(infinite speed of light) and semi-classical (vanishing Planck constant) regimes.
	This unified scaling allows us to justify the simultaneous non-relativistic and semi-classical limit,
	while retaining the nontrivial influence of the Chern--Simons gauge structure.
	Using a modulated energy method, we establish quantitative convergence rates
	toward the corresponding compressible Euler--Chern--Simons system as the scaling parameter tends to zero.
\end{abstract}
	\maketitle
	
\section{Introduction}\label{sec:1}
\setcounter{equation}{0}	

Planar physics investigates phenomena confined to two spatial dimensions and exhibits distinctive behaviors that are not captured by the usual $(1+3)$-dimensional classical and quantum electrodynamics. In a $(1+2)$-dimensional spacetime setting, the Chern--Simons gauge theory provides an alternative to Maxwell theory and plays an important role in the effective description of various planar quantum phenomena. In particular, it is widely regarded as a natural framework for the fractional quantum Hall effect and the emergence of anyonic statistics \cite{EHI92}. Among the models arising from Chern--Simons gauge theory, we consider the Chern--Simons--Higgs (CSH) model, which governs the dynamics of relativistic charged scalar fields coupled to a self-consistent gauge field in $(1+2)$-dimensional Minkowski spacetime. The CSH model was originally introduced in the study of vortex solutions in abelian Chern--Simons theories~\cite{HKP90,JW90}, and has since become a basic model for investigating topological and nontopological solitons, gauge interactions, and planar quantum dynamics \cite{D95}.

	To describe the dynamics of the CSH model, we work in relativistic coordinates
	$x_\mu=(x_0,x_1,x_2)=(ct,x_1,x_2)$ endowed with the Minkowski metric
	$\mathrm{diag}(+1,-1,-1)$. The model is governed by the Lagrangian density
	\[
	\mathcal{L}_{\textup{CSH}}
	:= \frac{\kappa}{2c}\epsilon^{\rho\mu\nu}A_{\rho}F_{\mu\nu}
	+ \hbar^2 D_{\mu}\phi\,\overline{D^{\mu}\phi}
	- m^2c^2|\phi|^2
	- V(2m|\phi|^2),
	\]
	where $\phi:\mathbb{R}^{1+2}\to\mathbb{C}$ is the complex scalar field,
	$A_\mu:\mathbb{R}^{1+2}\to\mathbb{R}$ ($\mu=0,1,2$) is the gauge potential, and
	$F_{\mu\nu}:=\partial_\mu A_\nu-\partial_\nu A_\mu$ denotes the field tensor.
	The covariant derivative is defined by
	$D_\mu:=\partial_\mu-\frac{\i}{c\hbar}A_\mu$, with $\partial_0:=\frac{1}{c}\partial_t$ and $\partial_j:=\partial_{x_j}$ for $j=1,2$.
	The Levi--Civita symbol $\epsilon^{\rho\mu\nu}$ is normalized by $\epsilon^{012}=1$.
	The constants $c$, $\hbar$, $m$, and $\kappa$ denote the speed of light, Planck's constant, the particle mass, and the Chern--Simons coupling parameter, respectively. Throughout this work, we consider a power-type self-interaction potential
	$V(\rho)=\frac{1}{\gamma-1}\rho^\gamma,\gamma>1,$
	and adopt the Einstein summation convention for repeated indices: Greek indices range over $0,1,2$ and Latin indices over $1,2$.
	Then, the Euler--Lagrange equations associated with $\mathcal{L}_{\textup{CSH}}$
	lead to the CSH system
	\begin{align}
		\begin{aligned}\label{CSH}
			&\hbar^2 D_\mu D^\mu \phi + m^2 c^2 \phi
			+ 2m V'(2m|\phi|^2)\phi = 0,\\
			&\kappa\Big(\frac{1}{c}\partial_t A_1-\partial_1 A_0\Big)
			=-2\hbar\,\Im(\overline{\phi}D_2\phi),\qquad
			\kappa\Big(\frac{1}{c}\partial_t A_2-\partial_2 A_0\Big)
			=2\hbar\,\Im(\overline{\phi}D_1\phi),\\
			&\frac{\kappa}{c}\big(\partial_1A_2-\partial_2A_1\big)
			=\frac{2\hbar}{c}\,\Im(\overline{\phi}D_0\phi).
		\end{aligned}
	\end{align}

	Since its introduction, the CSH system has been extensively studied from a mathematical perspective. 
	The global well-posedness under the Coulomb gauge and mild assumptions on the self-interaction potential was first established in \cite{CC02}. 
	Subsequently, local and global well-posedness for low-regularity solutions under the Coulomb, Lorenz, and temporal gauge conditions were obtained in \cite{B09,H07,H11,HO16,ST13,Y11}. 
	Beyond Cauchy problems, vortex solutions and their properties were investigated in \cite{CY95,Y01}, and non-relativistic limits toward the Chern--Simons--Schr\"odinger (CSS) system were analyzed in \cite{CH09,HS07,HM26}.
	
	More broadly, asymptotic limits of relativistic and quantum systems have been studied in a wide range of contexts. 
	Non-relativistic limits for relativistic equations such as the Klein--Gordon and Dirac equations have a long history; see, for instance, \cite{S79}. 
	For gauged quantum models, non-relativistic limits of Maxwell-gauged systems toward Schr\"odinger-type models have been studied in \cite{BMS04,JS21,MN03}. 
	Independently, semi-classical limits of Schr\"odinger-type equations and Schr\"odinger--Poisson systems have been studied; see, for example, \cite{AC07,BMP01,JW03,LL03,P02}, as well as the review article \cite{JMS11}. 
	More recently, semi-classical limits for the Chern--Simons--Schr\"odinger and Maxwell--Schr\"odinger systems have been established in \cite{KM22,KM24}. 
	While these two limiting regimes have been widely studied separately, a simultaneous non-relativistic and semi-classical limit remains comparatively less developed in the existing literature, with the exception of \cite{KM25,LW12}.
	
	In this work, we investigate the simultaneous non-relativistic and semi-classical limit of the CSH system and rigorously justify its convergence toward a compressible Euler-type hydrodynamic system with explicit rates. 
	Our result provides a unified framework that is consistent with the previously established non-relativistic limit from the CSH system to the CSS system and the semi-classical limit from the CSS system to an Euler-type hydrodynamic system, while capturing both limits within a single simultaneous limiting process. To achieve this, we introduce a suitable modulation of the scalar field given by
	\begin{equation*}
		\psi(t,x):=\sqrt{2m}\phi(t,x)\exp\left(\frac{\i mc^2 t}{\hbar}\right),
	\end{equation*}
	which removes the fast relativistic oscillations. 
	We further introduce a single scaling parameter $\e$ that simultaneously controls the speed of light and Planck's constant. 
	More precisely, we set, for $\delta>0$,
	\[
	c^{-1}=\e^{\delta},\qquad \hbar=\e.
	\]

	Under this scaling, we obtain the following one-parameter family of modulated CSH system:
	\begin{align}\label{CSH_parameter_1}
		\begin{aligned}
			&\i \e D^\e_t\psi^\e-\frac{\e^{2+2\delta}}{2}D^\e_tD^\e_t \psi^\e
			+\frac{\e^2}{2} \left(D^\e_1D^\e_1+D^\e_2D^\e_2\right)\psi^\e
			- V'\!\left(|\psi^\e|^2\right)\psi^\e=0,\\
			&\e^{\delta}\partial_tA^\e_1-\partial_1 A^\e_0=-\e\Im(\overline{\psi^\e}D^\e_2\psi^\e),\quad
			\e^{\delta}\partial_tA^\e_2-\partial_2 A^\e_0=\e\Im(\overline{\psi^\e}D^\e_1\psi^\e),\\
			&\e^{\delta}(\partial_1A^\e_2-\partial_2 A^\e_1)
			=-|\psi^\e|^2 + \e^{1+2\delta}\Im\left(\overline{\psi^\e} D^\e_t\psi^\e\right),
		\end{aligned}
	\end{align}
	where
	\[
	D^\e_t := \pa_t-\i\e^{-1}A^\e_0,\qquad 
	D^{\e}_j:=\pa_j-\i\e^{-1+\delta}A^{\e}_j,
	\]
	and the superscript $\e$ denotes dependency on the scaling parameter. 
	The precise derivation of \eqref{CSH_parameter_1} and its basic properties are presented in Sections~\ref{sec:2.2} and \ref{sec:2.3}.

	At the hydrodynamic level, by introducing the macroscopic density $\rho^\e$, momentum $J^\e$, and the relativistic correction $\rho^\e_R$, we formally obtain the following relativistic quantum hydrodynamic system:
	\begin{align}
		\begin{aligned}\label{CSH-hydro_1}
			&\pa_t(\rho^\e-\rho^\e_R)+\nabla\cdot(\rho^\e u^\e) = 0,\\
			&\pa_t((\rho^\e-\rho^\e_R)u^\e)+\nabla\cdot (\rho^\e u^\e\otimes u^\e)
			+\nabla p\left(\rho^\e\right)
			=\frac{\e^2\rho^\e}{2}\nabla\left(\frac{\square_\delta\sqrt{\rho^\e}}{\sqrt{\rho^\e}}\right),\\
			&\pa_t (\e^{\delta}A^{\e})-\nabla A^\e_0=(\rho^\e u^\e)^{\perp},\quad 
			\nabla \times (\e^{\delta}A^\e)=-\rho^\e+\rho^\e_R,
		\end{aligned}
	\end{align}
	where $p(\rho):=\rho V'(\rho)-V(\rho)=\rho^\gamma$ and $\square_\delta:=-\e^{2\delta}\pa_t^2+\Delta$. 
	Here, $A^\e=(A^\e_1,A^\e_2)$ and $D^\e=(D^\e_1,D^\e_2)$ denote the spatial gauge potential and the spatial covariant derivative, respectively.
	The precise derivation of \eqref{CSH-hydro_1} is given in Section~\ref{sec:2.3}. As $\e\to0$, the system \eqref{CSH-hydro_1} formally converges to the compressible Euler equations coupled with Chern--Simons gauge fields,
	\begin{align}
		\begin{aligned}\label{Euler-CS_1}
			&\pa_t \rho +\nabla\cdot(\rho u) = 0,\\
			&\pa_t (\rho u) +\nabla\cdot(\rho u \otimes u) +\nabla p(\rho) =0,\\
			&\pa_t A-\nabla A_0=(\rho u)^{\perp},\quad \nabla \times A=-\rho.
		\end{aligned}
	\end{align}
	We refer to \eqref{Euler-CS_1} as the Euler--Chern--Simons (Euler--CS) system. 
	This system was introduced as a semi-classical approximation of the CSS system in \cite{KM22}.

	Therefore, our goal in the present work is a rigorous and quantitative derivation of the \emph{hydrodynamic limit} from the modulated CSH system \eqref{CSH_parameter_1} toward the Euler--CS system \eqref{Euler-CS_1} as $\e\to0$ in the simultaneous non-relativistic and semi-classical regime. We defer the exact statement of the main theorem to  Section~\ref{sec:2.4}, 
	as several preliminaries are required to formulate our results.
	
	Our strategy for the asymptotic analysis is based on modulated energy estimates. 
	Modulated energy (closely related to the relative entropy method in kinetic theory and fluid mechanics) provides a quasi-metric between a wave or kinetic model and its macroscopic limit. It has been successfully applied to numerous asymptotic problems, including the Boltzmann equation \cite{BGL00}, the Vlasov--Poisson equation \cite{B00}, and the Vlasov--Navier--Stokes system \cite{MV08}. 
	Our previous works \cite{KM22,KM23a,KM24,KM23b} on semi-classical limits for the gauged 
	Schr\"odinger equations also rely on this approach. In the present work, we further develop this framework to derive the hydrodynamic limit 
	of the CSH system under the simultaneous scaling described above.
	
    The remainder of the paper is organized as follows.
    In Section~\ref{sec:2}, we review known results on the CSH system and related models, and present the derivation of the modulated CSH system together with the associated conservation laws and their hydrodynamic structure.
    Section~\ref{sec:3} is devoted to quantitative estimates for the modulated energy, which constitute the key analytic ingredient in the proof of the hydrodynamic limit.
    Based on these estimates, Section~\ref{sec:4} establishes the hydrodynamic limit of the CSH system with a quantitative convergence rate.
    Finally, Section~\ref{sec:5} concludes the paper with a summary of the main results and a discussion of possible extensions and future directions.

\section{Preliminaries}\label{sec:2}
\setcounter{equation}{0}
	In this section, we introduce the necessary preliminaries for the CSH system and present the derivation of the associated modulated system.
	In particular, we collect the conservation laws of the CSH system, which play a fundamental role in the subsequent analysis, and describe their hydrodynamic structure.
	This section also contains the precise statement of our main theorem.

\subsection{The Cauchy problem for the CSH system}\label{sec:2.1}
	To investigate the hydrodynamic limit, we first recall the well-posedness theory for the CSH system \eqref{CSH}.
	The CSH system is gauge invariant under the transformation
	\[\phi\to\phi e^{-\i\chi},\quad A_\mu \to A_\mu -c\hbar \pa_\mu \chi,\]
	for any smooth function $\chi : \mathbb{R}^{1+2} \to \mathbb{R}$.
	As a consequence, the well-posedness of the Cauchy problem requires the imposition of a suitable gauge condition. 
	Among various choices, we work under the \emph{Coulomb gauge} condition $\nabla \cdot A = 0$, which allows one to exploit null structures and elliptic features that are particularly useful for the analysis of the system \cite{H07,H11}.
	When a gauge condition is imposed, one must also address the issue of over-determination.
	In the present setting, this issue is resolved by observing that the time evolution of the constraint equation $\eqref{CSH}_3$ is preserved by the remaining equations in \eqref{CSH}.
	More precisely, one observes that
	\begin{align*}
		\p_t\left(\frac{\kappa}{c}(\p_1A_2-\p_2A_1)-\frac{2\hbar}{c}\Im(\overline{\phi}D_0\phi)\right)=0.
	\end{align*}
	Therefore, the constraint $\eqref{CSH}_3$ can be consistently imposed as part of the initial data.
	
	Accordingly, under the Coulomb gauge condition, the Cauchy problem for the CSH system \eqref{CSH} can be reformulated as follows:
	\begin{align}\label{eq2.1}
		\begin{aligned}
			&\hbar^2(D_0D_0-D_1D_1-D_2D_2)\phi +m^2c^2\phi+2mV'(2m|\phi|^2)\phi=0,\\
			&\Delta A_0 =\frac{2\hbar}{\kappa} \left(\p_1\Im(\overline{\phi}D_2\phi)-\p_2\Im(\overline{\phi}D_1\phi)\right),\\
			&\Delta A_1=-\frac{2\hbar}{\kappa}\p_2\Im(\overline{\phi}D_0\phi),\quad
			\Delta A_2=\frac{2\hbar}{\kappa}\p_1\Im(\overline{\phi}D_0\phi).
		\end{aligned}
	\end{align} 
	The system is supplemented with the initial data
	\begin{align}\label{eq2.2}
		\phi_0(x)=\phi(0,x),\quad \phi_1(x)=\p_t\phi(0,x),\quad a_{j}(x)=A_{j}(0,x),
	\end{align}
	together with the constraints
	\begin{align}\label{eq2.3}
		\p_1a_1+\p_2a_2=0,\quad 
		\frac{\kappa}{c}(\p_1a_2-\p_2a_1)=\frac{2\hbar}{c^2} \Im(\overline{\phi}_0\phi_1)-\frac{2}{c^2}a_0|\phi_0|^2,
	\end{align}
	where $a_0(x)=A_0(x,0)$ is determined by
	\begin{align*}
		\Delta a_0 =\frac{2\hbar}{\kappa}\Big(\pa_1\Im\big(\overline{\phi_0}\pa_2\phi_0-\frac{\i}{c\hbar}a_2|\phi_0|^2\big)
		-\pa_2\Im\big(\overline{\phi_0}\pa_1\phi_0-\frac{\i}{c\hbar}a_1|\phi_0|^2\big)\Big).
	\end{align*}
	
	The following result guarantees the global well-posedness of the Cauchy problem \eqref{eq2.1}--\eqref{eq2.3} under the Coulomb gauge condition.
	
	\begin{theorem}\label{Thm_Well-posedness} \emph{\cite{CC02}}
		For initial data $(\phi_0,\phi_1,a_{1},a_2)\in H^2(\bbr^2)\times H^{1}(\bbr^2) \times H^1(\bbr^2)^2$ satisfying the constraints \eqref{eq2.3}, there exists a unique global solution to \eqref{eq2.1}--\eqref{eq2.3} such that
		\[
		\phi\in C([0,T_1];H^{2}(\bbr^2))\cap C^1([0,T_1];H^1(\bbr^2)),\quad 
		A_{1},\,A_2\in C([0,T_1];H^1(\bbr^2)),
		\]
		for any $T_1>0$. Moreover, the solution depends continuously on the initial data.
	\end{theorem}

	\begin{remark}
		The result in \cite{CC02} also guarantees the existence of solutions with higher regularity.
		However, for our purposes, the regularity provided by Theorem~\ref{Thm_Well-posedness} is sufficient.
		In particular, this level of regularity ensures that the modulated energy is well defined, which is the key analytic quantity in the derivation of the hydrodynamic limit.
		We also note that well-posedness results at lower regularity levels, including the existence of energy solutions, are well established in the literature; see, for instance, \cite{H07,H11}.
	\end{remark}

\subsection{The modulated CSH system}\label{sec:2.2}
	In this subsection, we derive the modulated CSH system associated with the modulated wave function
	\[\psi(t,x)=\sqrt{2m}\exp\left(\frac{\i mc^2t}{\hbar}\right)\phi(t,x).\]
	This modulation removes the leading-order rest-mass oscillations and is standard in the study of the non-relativistic limit.
	
	We begin by rewriting the CSH system \eqref{CSH} in the following expanded form:
	\begin{align}\label{CSH_expand}
		\begin{aligned}
			&\frac{\hbar^2}{c^2}D_tD_t\phi-\hbar^2(D_1D_1+D_2D_2)\phi+c^2m^2\phi+2mV'(2m|\phi|^2)\phi=0,\\
			&\kappa\left(\frac{1}{c}\partial_tA_1-\partial_1 A_0\right)=-2\hbar\Im(\overline{\phi}D_2\phi),\quad
			\kappa\left(\frac1c\partial_tA_2-\partial_2 A_0\right)=2\hbar\Im(\overline{\phi}D_1\phi),\\
			&\frac{\kappa}{c} (\partial_1A_2-\partial_2 A_1)=\frac{2\hbar}{c^2}\Im(\overline{\phi}D_t\phi).
		\end{aligned}
	\end{align}
	Here, we introduce $D_t := cD_0$ to make the scaling in $c$ explicit.
	Using the Leibniz rule for the covariant derivative,
	$D_t(\psi(t,x) g(t)) = (D_t\psi)g +\psi(\pa_t g),$
	we obtain the identities
	\begin{align}
		\begin{aligned}\label{C-1}
			&D_t\phi = \frac{1}{\sqrt{2m}}\exp\left(-\frac{\i mc^2t}{\hbar}\right)\left(D_t\psi-\frac{\i mc^2}{\hbar}\psi\right),\\
			&D_tD_t\phi =  \frac{1}{\sqrt{2m}}\exp\left(-\frac{\i mc^2t}{\hbar}\right)\left(D_tD_t\psi-\frac{2\i mc^2}{\hbar}D_t\psi-\frac{m^2c^4}{\hbar^2}\psi\right).
		\end{aligned}
	\end{align}
	Since the oscillatory factor $\exp\left(\dfrac{\i mc^2 t}{\hbar}\right)$ is independent of the spatial variables, we also obtain, for $j=1,2$,
	\begin{align}
		\begin{aligned}\label{C-2}
			&D_j\phi =  \frac{1}{\sqrt{2m}}\exp\left(-\frac{\i mc^2t}{\hbar}\right)D_j\psi,\quad 
			D_jD_j\phi =  \frac{1}{\sqrt{2m}}\exp\left(-\frac{\i mc^2t}{\hbar}\right)D_jD_j\psi.
		\end{aligned}
	\end{align}
	Substituting \eqref{C-1}--\eqref{C-2} into \eqref{CSH_expand}$_1$, we obtain the equation for $\psi$:
	\[
	2\i m \hbar D_t\psi-\frac{\hbar^2}{c^2}D_tD_t \psi 
	+\hbar^2\left(D_1D_1+D_2D_2\right)\psi 
	-2mV'\left(|\psi|^2\right)\psi=0.
	\]
	Moreover, inserting the above relations into \eqref{CSH}$_3$ yields
	\begin{align*}
		&\frac{2\hbar}{c^2}\Im\left(\overline{\phi}D_t\phi\right) 
		= -|\psi|^2 + \frac{\hbar}{mc^2}\Im\left(\overline{\psi} D_t\psi\right).
	\end{align*}
	Consequently, we derive the following modulated CSH system:
	\begin{align}\label{CSH_modulated}
		\begin{aligned}
			&\i\hbar D_t\psi-\frac{\hbar^2}{2mc^2}D_tD_t \psi 
			+\frac{\hbar^2}{2m}\left(D_1D_1+D_2D_2\right)\psi 
			- V'\left(|\psi|^2\right)\psi=0,\\
			&\kappa\left(\frac1c\partial_tA_1-\partial_1 A_0\right)
			=-\frac{\hbar}{m}\Im(\overline{\psi}D_2\psi),\quad
			\kappa\left(\frac1c\partial_tA_2-\partial_2 A_0\right)
			=\frac{\hbar}{m}\Im(\overline{\psi}D_1\psi),\\
			&\frac{\kappa}{c} (\partial_1A_2-\partial_2 A_1)
			=-|\psi|^2 + \frac{\hbar}{mc^2}\Im\left(\overline{\psi} D_t\psi\right).
		\end{aligned}
	\end{align}
	Unlike $c$ and $\hbar$, which govern the non-relativistic and semi-classical limits,
	the parameters $m$ and $\kappa$ are fixed positive constants that can be absorbed by rescaling and therefore do not affect the structure of the limiting system.
	Accordingly, we normalize these constants by setting $m=\kappa=1$.
	
	\begin{remark}
		Formally letting $c\to\infty$ in \eqref{CSH_modulated}, all terms of order $c^{-2}$ vanish and the system reduces to the Chern--Simons--Schr\"odinger (CSS) system, as originally discussed in \cite{JP90}. 
		In this limit, the rescaled spatial gauge potentials
		$\mathcal A_j := c^{-1}A_j$ remain nontrivial and satisfy
		\begin{align}
			\begin{aligned}\label{CSS}
				&\i\hbar D_t\psi
				+\frac{\hbar^2}{2m}\left(\mathcal{D}_1\mathcal{D}_1+\mathcal{D}_2\mathcal{D}_2\right)\psi
				- V'\left(|\psi|^2\right)\psi = 0,\\
				&\kappa \left(\pa_t\mathcal{A}_1-\pa_1 A_0\right)
				=-\frac{\hbar}{m} \Im(\overline{\psi}\mathcal{D}_2\psi),\quad
				\kappa \left(\pa_t\mathcal{A}_2-\pa_2 A_0\right)
				=\frac{\hbar}{m} \Im(\overline{\psi}\mathcal{D}_1\psi),\\
				&\pa_1\mathcal{A}_2-\pa_2\mathcal{A}_1=-|\psi|^2,
			\end{aligned}
		\end{align}
		where $\mathcal{D}_j:=\partial_j-\frac{\i}{\hbar}\mathcal{A}_j$. 
		In particular, both the Chern--Simons magnetic and electric fields persist in the CSS system. 
		Rigorous justifications of the non-relativistic limit from the CSH system to the CSS system can be found in \cite{CH09,HM26}.
	\end{remark}

\subsection{Conservation laws and hydrodynamic formulation}\label{sec:2.3}

	We establish several conservation laws for the modulated CSH system \eqref{CSH_modulated}, which play a fundamental role in revealing its hydrodynamic structure and in the subsequent convergence analysis.
	We now implement the simultaneous non-relativistic and semi-classical scaling introduced in the introduction by setting
	\[
	c^{-1}=\e^{\delta},\qquad \hbar=\e.
	\]
	Applying this scaling to \eqref{CSH_modulated}, we obtain the following one-parameter family of CSH systems:
	\begin{align}\label{CSH_parameter}
		\begin{aligned}
			&\i \e D^\e_t\psi^\e-\frac{\e^{2+2\delta}}{2}D^\e_tD^\e_t \psi^\e 
			+\frac{\e^2}{2} \left(D^\e_1D^\e_1+D^\e_2D^\e_2\right)\psi^\e 
			- V'\!\left(|\psi^\e|^2\right)\psi^\e=0,\\
			&\e^{\delta}\partial_tA^\e_1-\partial_1 A^\e_0=-\e\Im(\overline{\psi^\e}D^\e_2\psi^\e),\quad
			\e^{\delta}\partial_tA^\e_2-\partial_2 A^\e_0=\e\Im(\overline{\psi^\e}D^\e_1\psi^\e),\\
			&\e^{\delta}(\partial_1A^\e_2-\partial_2 A^\e_1)
			=-|\psi^\e|^2 + \e^{1+2\delta}\Im\left(\overline{\psi^\e} D^\e_t\psi^\e\right),
		\end{aligned}
	\end{align}
	where
	\[
	D^\e_t = \pa_t-\i\e^{-1}A^\e_0,\qquad 
	D^{\e}_j=\pa_j-\i\e^{-1+\delta}A^{\e}_j.
	\]
	
	The scaled system \eqref{CSH_parameter} satisfies the following conservation laws: mass, momentum, and energy. 
	Henceforth, we refer to \eqref{CSH_parameter} as the CSH system, which will be the main object of our analysis.
	
 	\begin{proposition}\label{P2.1}
 		Let $(\psi^\e, A_{\mu}^\e)$ be a solution to the CSH system \eqref{CSH_parameter}. 
 		Define the macroscopic density and momentum by
 		\[
 		\rho^\e:=|\psi^\e|^2,\quad 
 		J^\e:=\frac{\i\e}{2}\left(\psi^\e\overline{D^\e\psi^\e}-\overline{\psi^\e}D^\e\psi^\e\right)
 		=\e\Im(\overline{\psi^\e}D^{\e}\psi^\e),
 		\]
 		whose relativistic corrections are defined by
 		\begin{align*}
 			\rho^\e_R&:=\frac{\i\e^{1+2\delta}}{2}\left(\psi^\e\overline{D_t^\e\psi^\e}
 			-\overline{\psi^\e}D^\e_t\psi^\e\right)
 			=\e^{1+2\delta}\Im(\overline{\psi^\e}D^\e_t\psi^\e),\\
 			J^\e_R&:= \frac{\e^{2+2\delta}}{2}\left(D^\e_t\psi^\e\overline{D^{\e}\psi^\e}
 			+\overline{D^\e_t\psi^\e}D^{\e}\psi^\e\right)
 			=\e^{2+2\delta}\Re(\overline{D^\e_t\psi^\e} D^{\e}\psi^\e ).
 		\end{align*}
 		Furthermore, we define the total energy by
 		\begin{align}\label{energy}
 			\mathcal{E}^\e(t):=\int_{\bbr^2}\left(
 			\frac{\e^{2+2\delta}}{2}|D^\e_t\psi^\e|^2
 			+\frac{\e^2}{2}|D^{\e}\psi^\e|^2
 			+V\left(|\psi^\e|^2\right)\right)(t,x)\,\d x.
 		\end{align}
 		The CSH system \eqref{CSH_parameter} satisfies the following conservation laws:
 		\begin{enumerate}
 			\item Mass conservation law:
 			\[
 			\pa_t(\rho^\e-\rho^\e_R)+\nabla\cdot J^\e = 0.
 			\]
 			
 			\item Momentum conservation law:
 			\begin{align*}
 				\pa_t(J^\e-J^\e_R)
 				&+\frac{\e^2}{2}\nabla\cdot \Big(
 				D^\e\psi^\e\otimes \overline{D^\e\psi^\e}
 				+\overline{D^\e\psi^\e}\otimes D^\e\psi^\e
 				-\nabla \Re(\overline{\psi^\e} D^\e \psi^\e)
 				\Big)\\
 				&\qquad
 				+\frac{\e^{2+2\delta}}{2}\nabla\pa_t\Re(\overline{\psi^\e}D^\e_t\psi^\e)
 				+\nabla p\left(|\psi^\e|^2\right)=0.
 			\end{align*}
 			
 			\item Energy conservation law:
 			\[
 			\frac{\d\mathcal{E}^\e(t)}{\d t}=0.
 			\]
 		\end{enumerate}
 	\end{proposition}

	\begin{proof}
	The proof is lengthy and technical, and is therefore postponed to Appendix \ref{Appen:A}.
	\end{proof}
	\begin{remark}
		For clarity, we write $J^\e=\rho^\e u^\e$, where $u^\e$ is defined pointwise by
		\[
		u^\e(t,x):=
		\begin{cases}
			\dfrac{J^\e(t,x)}{\rho^\e(t,x)}, & \text{if } \rho^\e(t,x)\neq 0,\\
			0, & \text{if } \rho^\e(t,x)=0.
		\end{cases}
		\]
		Since
		\[
		|J^\e| = |\e\Im(\overline{\psi^\e}D^\e\psi^\e)|
		\leq |\psi^\e||\e D^\e\psi^\e|,
		\]
		it follows that $\rho^\e(t,x)=0$ implies $J^\e(t,x)=0$.
		Hence the identity $J^\e=\rho^\e u^\e$ is well defined everywhere.
	\end{remark}

	A classical approach to deriving hydrodynamic formulations of Schr\"odinger-type equations
	is to apply the Madelung transformation \cite{M27} to $\psi^\e$,
	\[\psi^\e(t,x) = \sqrt{\rho^\e(t,x)} \exp\left(\frac{\i}{\e} S^\e(t,x)\right).\]
	Rather than pursuing this route, we directly rewrite the momentum conservation law
	for the CSH system \eqref{CSH_parameter} in terms of hydrodynamic variables.
	Assume first that $\rho^\e\neq0$. Then we compute
		\begin{align*}
		\rho^\e_R u^\e &= \frac{\rho^\e_R}{\rho^\e}J^\e = \left(\frac{\i\e^{1+2\delta}}{2|\psi^\e|^2}(\psi^\e \overline{D^\e_t\psi^\e}-\overline{\psi^\e}D^\e_t\psi^\e)\right)\left( \frac{\i\e}{2}(\psi^\e\overline{D^\e\psi^\e}-\overline{\psi^\e}D^\e\psi^\e)\right)\\
		&=-\frac{\e^{2+2\delta}}{4|\psi^\e|^2}(\psi^\e \overline{D^\e_t\psi^\e}-\overline{\psi^\e}D^\e_t\psi^\e)(\psi^\e\overline{D^\e\psi^\e}-\overline{\psi^\e}D^\e\psi^\e)\\
		&=\frac{\e^{2+2\delta}}{2}(\overline{D^\e_t\psi^\e}D^\e\psi^\e+D^{\e}_t\psi^\e \overline{D^\e\psi^\e})-\frac{\e^{2+2\delta}}{4|\psi^\e|^2}(\psi^\e \overline{D^\e_t\psi^\e}+\overline{\psi^\e}D^\e_t\psi^\e)(\psi^\e\overline{D^\e\psi^\e}+\overline{\psi^\e}D^\e\psi^\e)\\
		& = J^\e_R -\frac{\e^{2+2\delta}}{|\psi^\e|^2}\Re(\overline{\psi^\e}D^\e_t\psi^\e)\Re(\overline{\psi^\e}D^\e\psi^\e)=J^\e_R -\frac{\e^{2+2\delta}}{4\rho^\e} (\pa_t\rho^\e)(\nabla\rho^\e).
		\end{align*}
	where we used
		\[\Re(\overline{\psi^\e}D^\e_t\psi^\e)=\Re(\overline{\psi^\e}\pa_t \psi^\e) =\frac12 \pa_t \rho^\e,\quad \Re(\overline{\psi^\e}D^\e\psi^\e) = \frac{1}{2}\nabla|\psi^\e|^2=\frac12\nabla \rho^\e.\]
	Moreover, we can formally verify that 	
		\begin{align*}
		\frac{\e^2}{2}&\nabla \cdot	(D^\e\psi^\e\otimes\overline{D^\e\psi^\e}+\overline{D^\e\psi^\e}\otimes D^\e\psi^\e)-\frac{\e^2}{2}\Delta\Re\left(\overline{\psi^\e} D^\e \psi^\e \right)\\
		&=\nabla\cdot(\rho^\e u^\e\otimes u^\e)+\frac{\e^2}{4}\nabla\cdot\left(\frac{1}{|\psi^\e|^2} D^\e|\psi^\e|^2\otimes D^\e|\psi^\e|^2\right)
		-\frac{\e^2}{4}\Delta\nabla|\psi^\e|^2\\
		&=\nabla\cdot(\rho^\e u^\e\otimes u^\e)+\frac{\e^2}{4}\nabla \cdot\left(\frac1{\rho^\e}\nabla \rho^\e\otimes \nabla \rho^\e\right)-\frac{\e^2}{4}\Delta\nabla\rho^\e\\
		&=\nabla\cdot(\rho^\e u^\e\otimes u^\e)-\frac{\e^2\rho^\e}{2}\nabla\left(\frac{\Delta \sqrt{\rho^\e}}{\sqrt{\rho^\e}}\right).
		\end{align*}
		Thus, the momentum equation in Proposition \ref{P2.1} can be rewritten as
		\begin{align*}
		\pa_t&((\rho^\e-\rho^\e_R)u^\e)+\nabla\cdot (\rho^\e u^\e\otimes u^\e) +\nabla p\left(\rho^\e\right) \\
		&= \frac{\e^{2+2\delta}}{4}\pa_t\left(\frac{\pa_t\rho^\e\nabla\rho^\e}{\rho^\e}\right)-\frac{\e^{2+2\delta}}{4}\pa_t\pa_t\nabla\rho^\e+\frac{\e^2\rho^\e}{2}\nabla\left(\frac{\Delta \sqrt{\rho^\e}}{\sqrt{\rho^\e}}\right)\\
		&=-\frac{\e^{2+2\delta}\rho^\e}{2} \nabla \left(\frac{\pa_t\pa_t\sqrt{\rho^\e}}{\sqrt{\rho^\e}}\right)+\frac{\e^2\rho^\e}{2}\nabla\left(\frac{\Delta \sqrt{\rho^\e}}{\sqrt{\rho^\e}}\right)= \frac{\e^2\rho^\e}{2}\nabla\left(\frac{\square_\delta\sqrt{\rho^\e}}{\sqrt{\rho^\e}}\right).
		\end{align*}
		
		In this manner, the formal quantum hydrodynamic system associated with the
		CSH system \eqref{CSH_parameter} is obtained from the mass and momentum conservation laws.
		Moreover, the gauge equations can be rewritten in hydrodynamic form,
		leading to the following system:
		\begin{align}
			\begin{aligned}\label{CSH-hydro}
				&\pa_t(\rho^\e-\rho^\e_R)+\nabla\cdot(\rho^\e u^\e) = 0,\\
				&\pa_t((\rho^\e-\rho^\e_R)u^\e)
				+\nabla\cdot (\rho^\e u^\e\otimes u^\e)
				+\nabla p\left(\rho^\e\right)
				=\frac{\e^2\rho^\e}{2}\nabla\left(\frac{\square_\delta\sqrt{\rho^\e}}{\sqrt{\rho^\e}}\right),\\
				&\pa_t (\e^{\delta}A^{\e})-\nabla A^\e_0=(\rho^\e u^\e)^{\perp},\quad 
				\nabla \times (\e^{\delta}A^\e)=-\rho^\e+\rho^\e_R,
			\end{aligned}
		\end{align}
		where $p(\rho)=\rho V'(\rho)-V(\rho)= \rho^\gamma$ is an isentropic pressure law, and
		$\square_\delta=- \e^{2\delta}\pa_t^2+ \Delta$ denotes the d'Alembertian operator.
		Formally, the hydrodynamic system \eqref{CSH-hydro} reduces, as $\e\to0$, to the compressible Euler--CS system		
		\begin{align}
			\begin{aligned}\label{Euler-CS}
				&\pa_t \rho +\nabla\cdot(\rho u) = 0,\\
				&\pa_t (\rho u) +\nabla\cdot(\rho u \otimes u) +\nabla p(\rho) =0,\\
				&\pa_t A-\nabla A_0=(\rho u)^{\perp},\quad 
				\nabla \times A=-\rho,
			\end{aligned}
		\end{align}
		which constitutes the hydrodynamic limit investigated in this work.
		
	\begin{remark}
	In the relativistic quantum hydrodynamic formulation \eqref{CSH-hydro}, by neglecting terms of order $c^{-2}=\e^{2\delta}$ in the non-relativistic regime, one formally obtains the hydrodynamic formulation associated with the CSS system \eqref{CSS}. 
	Within this framework, the pure semi-classical limit from the CSS system to the Euler--CS system \eqref{Euler-CS} was established in \cite{KM22}. 
		
	In contrast, the present work derives the Euler--CS system directly from the relativistic CSH system through a simultaneous non-relativistic and semi-classical scaling. 
	Thus, our result unifies the non-relativistic limit from CSH to CSS and the semi-classical limit from CSS to Euler--CS
	within a single framework.
	\end{remark}
	
\subsection{Statement of the main theorem}\label{sec:2.4}
	
	We are now ready to state the main theorem of this paper concerning the hydrodynamic limit
	from the CSH system \eqref{CSH_parameter} to the Euler--CS system \eqref{Euler-CS}.
    A fundamental tool in our analysis is the \emph{modulated energy},
	which plays a role analogous to the notion of \emph{relative entropy}
	in the hydrodynamic limit of the kinetic equations; see, for instance, \cite{B00,S09}.
	We define the modulated energy $\mathcal{H}^\e$ associated with \eqref{CSH_parameter} by
	\[\mathcal{H}^\e(t) := \int_{\mathbb{R}^2} \frac{1}{2} |(\e D^\e - \i u)\psi^\e|^2
	+ \frac{1}{2} |\e^{1+\delta} D^\e_t \psi^\e|^2+ \frac{1}{\gamma-1} p(|\psi^\e|^2 \mid \rho)\, \d x.\]
	Here the function $p(n | \rho)$ is defined by
	\[
	p(n | \rho) :=  n^\gamma - \rho^\gamma - \gamma \rho^{\gamma-1} (n - \rho),
	\]
	and represents the relative internal energy associated with the pressure law.
	
	To establish the hydrodynamic limit, we impose a \emph{well-prepared initial data} assumption,
	which ensures that the initial data for the CSH system \eqref{CSH_parameter}
	and the Euler--CS system \eqref{Euler-CS} are compatible at the level of the modulated energy. 
	We also impose Coulomb-type constraints to fix the gauge.

	\medskip
	\noindent $\bullet$ ($\mathcal{C}1$):
	The initial data $(\psi^\e_{\textup{in}}, A^{\e}_{0,\textup{in}}, A^{\e}_{\textup{in}})$
	for the CSH system \eqref{CSH_parameter} and
	$(\rho_{\textup{in}}, u_{\textup{in}}, A_{0,\textup{in}}, A_{\textup{in}})$
	for the Euler--CS system \eqref{Euler-CS} satisfy 
	\begin{align*}
	\mathcal{H}^\e(0)
	= \int_{\mathbb{R}^2}
	\frac12 | (\e D^\e - \i u_{\textup{in}})\psi^\e_{\textup{in}} |^2
	+ \frac12 |\e^{1+\delta} D^\e_t \psi^\e_{\textup{in}} |^2
	+ \frac{1}{\gamma-1} p(|\psi^\e_{\textup{in}}|^2 \mid \rho_{\textup{in}})
	\,\d x
	\le C \e^\lambda,
	\end{align*}
	for some $\lambda>0$.
	
	\medskip
	\noindent $\bullet$ ($\mathcal{C}2$):
	The initial gauge fields satisfy the Coulomb-type constraints
	\begin{align*}
	&\nabla \cdot A^\e_{\textup{in}} = 0, \qquad
	\nabla \times (\e^{\delta} A^\e_{\textup{in}})
	= -|\psi^\e_{\textup{in}}|^2
	+ \e^{1+2\delta} \Im(\overline{\psi^\e_{\textup{in}}} D^\e_t \psi^\e_{\textup{in}}),\\
	&\nabla \cdot A_{\textup{in}} = 0, \qquad
	\nabla \times A_{\textup{in}} = -\rho_{\textup{in}} .
	\end{align*}

	\begin{theorem}\label{Thm:main}
	Let $(\psi^\e, A^\e_0, A^\e)$ be the global solution to the CSH system \eqref{CSH_parameter}, 
	and let $(\rho,u,A_0, A)$ be the local smooth solution to the Euler--CS system \eqref{Euler-CS} on $[0,T_*)$, 
	subject to the initial data $(\psi^\e_{\textup{in}},A^{\e}_{0,\textup{in}},A^{\e}_{\textup{in}})$ 
	and $(\rho_{\textup{in}},u_{\textup{in}},A_{0,\textup{in}},A_{\textup{in}})$ 
	satisfying assumptions $(\mathcal{C}1)$ and $(\mathcal{C}2)$, respectively.

	\begin{enumerate}
		\item Suppose $\gamma \ge 2$. Then the following convergences hold:
		\begin{align*}
			&\rho^\e \to \rho\quad\mbox{in}\quad L^\infty([0,T_*);L^\gamma(\bbr^2)),\\
			&J^\e \to \rho u\quad\mbox{in}\quad L^\infty([0,T_*);L^\frac{2\gamma}{\gamma+1}(\bbr^2)),\quad \sqrt{\rho^\e}u^\e \to \sqrt{\rho}u\quad\mbox{in}\quad L^\infty([0,T_*);L^2(\bbr^2)),\\
			&\rho^\e_R \to 0\quad\mbox{in}\quad L^\infty([0,T_*);L^{\frac{2\gamma}{\gamma+1}}(\bbr^2)),\quad J^\e_R \to 0\quad\mbox{in}\quad L^\infty([0,T_*);L^1(\bbr^2)),\\
			&A_0^\e \to A_0\quad \mbox{in}\quad L^\infty([0,T_*);L^{2\gamma}(\bbr^2)),\quad
	\nabla A_0^\e \to \nabla A_0\quad\mbox{in}\quad L^\infty([0,T_*);L^\frac{2\gamma}{\gamma+1}(\bbr^2)),\\
			&\e^\delta A^\e \to A \quad\mbox{in}\quad L^\infty([0,T_*);L^\frac{2\gamma}{\gamma+1}(\bbr^2)).
		\end{align*}
		\item Suppose $1<\gamma<2$. Then the following convergences hold:
		\begin{align*}
			&\rho^\e \to \rho\quad\mbox{in}\quad L^\infty([0,T_*);L_{\textup{loc}}^\gamma(\bbr^2)),\\
			&J^\e \to \rho u\quad\mbox{in}\quad L^\infty([0,T_*);L_{\textup{loc}}^\frac{2\gamma}{\gamma+1}(\bbr^2)),\quad \sqrt{\rho^\e}u^\e \to \sqrt{\rho}u\quad\mbox{in}\quad L^\infty([0,T_*);L_{\textup{loc}}^2(\bbr^2)),\\
			&\rho^\e_R \to 0\quad\mbox{in}\quad L^\infty([0,T_*);L^{\frac{2\gamma}{\gamma+1}}(\bbr^2)),\quad J^\e_R \to 0\quad\mbox{in}\quad L^\infty([0,T_*);L^1(\bbr^2)),\\		
			&A_0^\e \to A_0\quad \mbox{in}\quad L^\infty([0,T_*);L_{\textup{loc}}^{2\gamma}(\bbr^2)),\quad
			\nabla A_0^\e \to \nabla A_0\quad\mbox{in}\quad L^\infty([0,T_*);L_{\textup{loc}}^\frac{2\gamma}{\gamma+1}(\bbr^2)),\\
			&\e^\delta A^\e \to A \quad\mbox{in}\quad L^\infty([0,T_*);L_{\textup{loc}}^\frac{2\gamma}{\gamma+1}(\bbr^2)).
		\end{align*}
	\end{enumerate}
	\end{theorem}

		\begin{remark}
		The local-in-time existence and regularity of smooth solutions ($H^s$ with sufficiently large $s$) 
		to the compressible Euler equations $\eqref{Euler-CS}_{1,2}$ are
		 standard; 
		see, for instance, \cite{L98, M84}. 
		Therefore, in this theorem, we consider a local-in-time smooth solution 
		$\rho, u \in C([0,T_*);H^s)$ with $s>3$, where $T_*$ denotes the lifespan of the Euler equations. 
		The gauge potentials are then recovered from $\eqref{Euler-CS}_3$ together with the assumption $(\mathcal{C}2)$.
	\end{remark}

	\begin{remark}
	At this stage, we establish only qualitative convergence of the hydrodynamic quantities. Quantitative convergence rates for the
	hydrodynamic limit, depending on the parameters $\delta$, $\lambda$, and $\gamma$, will be provided in Section~\ref{sec:4}.
	\end{remark}

	We outline the strategy for proving Theorem~\ref{Thm:main}. The first step is to estimate the modulated energy $\mathcal{H}^\e$, showing that it satisfies  
	\begin{equation}\label{est-modulated-energy}
	\sup_{0\leq t\leq T_*} \mathcal{H}^\e(t)\leq C\e^\alpha, \quad\mbox{where}\quad \alpha = \min\{1,\delta,\lambda\}.
	\end{equation}
	Once this bound is established, the convergence results follow from standard estimates. 
	Therefore, obtaining \eqref{est-modulated-energy} is the central step in the analysis of the non-relativistic and semi-classical limits of the CSH system.  Its proof will be given in Section~\ref{sec:3}, while Section~ \ref{sec:4} is devoted to convergence estimates.

\section{Modulated Energy Estimates for the CSH System}\label{sec:3}
\setcounter{equation}{0}

In this section, we derive a priori estimates for the modulated energy $\mathcal{H}^\e$ associated with the CSH system \eqref{CSH_parameter}, which constitute the core analytic ingredient in the proof of the main theorem.
Our approach is based on a careful decomposition of the modulated energy and the introduction of a suitable relativistic correction functional.
This strategy is inspired by the modulated energy method developed in \cite{LW12}, adapted here to the CSH setting.

We begin by recalling the definition of the modulated energy $\mathcal{H}^\e$ and explaining its relation to the total energy $\mathcal{E}^\e$ defined in \eqref{energy} for the CSH system \eqref{CSH_parameter}.
Expanding the first term in $\mathcal{H}^\e$, we obtain
\begin{align*}
	\int_{\bbr^2}\frac{1}{2}|(\e D^\e-\i u)\psi^\e|^2\,\d x 
	&= \int_{\bbr^2}\frac{1}{2}|\e D^\e\psi^\e|^2 - \frac{\i\e}{2}(\psi^\e\overline{D^\e\psi^\e}-\overline{\psi^\e}D^\e\psi^\e)\cdot u+\frac{1}{2} |\psi^\e|^2|u|^2\,\d x\\
	&=\int_{\bbr^2}\frac{1}{2}|\e D^\e\psi^\e|^2\,\d x-\int_{\bbr^2}J^\e\cdot u\,\d x+\int_{\bbr^2}\frac{1}{2}\rho^\e|u|^2\,\d x.
\end{align*}
Consequently, we obtain the identity
\begin{align}\label{energy_modulated}
	\begin{aligned}
		\mathcal{H}^\e(t)&=\int_{\bbr^2}\frac{1}{2}|(\e D^\e-\i u)\psi^\e|^2\,\d x+\int_{\bbr^2}\frac{1}{2}|\e^{1+\delta} D^\e_t\psi^\e|^2\,\d x +\int_{\bbr^2}\frac{1}{\gamma-1}p\left(\rho^\e|\rho\right) \,\d x\\
		&= \mathcal{E}^\e(t) -\int_{\bbr^2}  J^\e \cdot u \,\d x +\int_{\bbr^2}\frac{1}{2}\rho^\e|u|^2\,\d x+\int_{\bbr^2}\left(\rho-\frac{\gamma}{\gamma-1}\rho^\e\right)\rho^{\gamma-1}\,\d x.
	\end{aligned}
\end{align}
In addition, we introduce the following {\it relativistic correction functional}:
\begin{equation}\label{correction}
	\mathcal{R}^\e(t):=\int_{\bbr^2} J^\e_R \cdot u \,\d x  +\int_{\bbr^2}\frac{\e^{2+2\delta}}{4}\pa_t \rho^\e \nabla \cdot u \,\d x-\int_{\bbr^2}\frac{1}{2}\rho^\e_R|u|^2\,\d x +\frac{\gamma}{\gamma-1}\int_{\bbr^2}\rho^\e_R\rho^{\gamma-1}\,\d x,
\end{equation}
which is an auxiliary quantity introduced to absorb additional relativistic remainder terms arising in the time derivative of the modulated energy; see the proof of Lemma~\ref{lem3.1}. We now establish estimates for $\mathcal{H}^\e$ and $\mathcal{R}^\e$.

\begin{proposition}\label{P3.1}
	Assume $\gamma>1$. Let $(\psi^\e, A^\e_0, A^\e)$ be the global solution to the CSH system \eqref{CSH_parameter}, 
	and let $(\rho,u,A_0, A)$ be the local smooth solution to the Euler--CS system \eqref{Euler-CS} on $[0,T_*)$, 
	subject to the initial data $(\psi^\e_{\textup{in}},A^{\e}_{0,\textup{in}},A^{\e}_{\textup{in}})$ 
	and $(\rho_{\textup{in}},u_{\textup{in}},A_{0,\textup{in}},A_{\textup{in}})$ 
	satisfying assumptions $(\mathcal{C}1)$ and $(\mathcal{C}2)$, respectively. Then, the following estimate holds:
	\begin{equation}\label{modulated_auxiliary_est}
		\frac{\d}{\d t}(\mathcal{H}^\e(t)+\mathcal{R}^\e(t)) \le C\mathcal{H}^\e (t)+C\e^{\min\{1,\delta\}},\quad \text{for} \quad 0\le t\le T_*.
	\end{equation}
\end{proposition}
To prove Proposition \ref{P3.1}, we first present the following lemma, which is used to derive the desired bound for the time derivative of $\mathcal{H}^\e + \mathcal{R}^\e$.
\begin{lemma}\label{lem3.1}
		Assume $\gamma>1$. Let $(\psi^\e, A^\e_0, A^\e)$ be the global solution to the CSH system \eqref{CSH_parameter}, 
		and let $(\rho,u,A_0, A)$ be the local smooth solution to the Euler--CS system \eqref{Euler-CS} on $[0,T_*)$, 
		subject to the initial data $(\psi^\e_{\textup{in}},A^{\e}_{0,\textup{in}},A^{\e}_{\textup{in}})$ 
		and $(\rho_{\textup{in}},u_{\textup{in}},A_{0,\textup{in}},A_{\textup{in}})$ 
		satisfying assumptions $(\mathcal{C}1)$ and $(\mathcal{C}2)$, respectively. Then, for $0\leq t\leq T_*$, we have
	\begin{align}
		\begin{aligned}\label{modulated-energy-est}
	&	\frac{\d}{\d t}(\mathcal{H}^\e(t)+\mathcal{R}^\e(t))\\
	&\qquad= -\int_{\bbr^2} \bigg[\frac{\e^2}{2}(D^\e\psi^\e\otimes\overline{D^\e\psi^\e}+\overline{D^\e\psi^\e}\otimes D^\e\psi^\e)-J^\e\otimes u-u\otimes J^\e+\rho^\e u\otimes\bigg]:\nabla u\,\d x \\
		&\qquad\quad-\int_{\bbr^2}\frac{\e^2}{4}\nabla 
		\rho^\e \cdot(\Delta u)+\left((\rho^\e)^\gamma-\rho^\gamma-\gamma\rho^{\gamma-1}(\rho^\e-\rho)\right)(\nabla\cdot u)\,\d x\\
		&\qquad\quad+\int_{\bbr^2} J^\e_R\cdot \pa_t u+\frac{\e^{2+2\delta}}{4}\pa_t\rho^\e  \left(\nabla \cdot \pa_t u\right)-\rho^\e_R u\cdot\pa_t u+\frac{\gamma}{\gamma-1}\rho^\e_R\pa_t(\rho^{\gamma-1})\,\d x.
		\end{aligned}
	\end{align}
\end{lemma}

\begin{proof}
By the conservation of total energy, the time derivative of the modulated energy in \eqref{energy_modulated} can be written as
\begin{align*}
	\frac{\d \mathcal{H}^\e(t)}{\d t} &= -\frac{\d}{\d t} \int_{\bbr^2} J^\e \cdot u \, \d x + \frac{\d}{\d t} \int_{\bbr^2} \frac{1}{2} \rho^\e |u|^2 \, \d x + \frac{\d}{\d t} \int_{\bbr^2}  \left( \rho - \frac{\gamma}{\gamma - 1} \rho^\e \right) \rho^{\gamma - 1} \, \d x=: \sum_{\ell=1}^3 \mathcal{I}_{1\ell}(t).
\end{align*}
We now estimate each term $ \mathcal{I}_{1\ell} $ for $ \ell=1,2,3 $ separately. \\

\noindent $\bullet$ (Estimate of $ \mathcal{I}_{11} $): We split the estimate of $ \mathcal{I}_{11} $ as
\begin{align*}
	\mathcal{I}_{11} &= -\frac{\d}{\d t} \int_{\bbr^2} J^\e \cdot u \, \d x 
	= -\int_{\bbr^2} \pa_t J^\e \cdot u \, \d x - \int_{\bbr^2} J^\e \cdot (\pa_t u) \, \d x =: \mathcal{I}_{111} + \mathcal{I}_{112}.
\end{align*}
Using the momentum equation in Proposition \ref{P2.1} (2), we estimate $ \mathcal{I}_{111} $ as
\begin{align*}
	\mathcal{I}_{111} &= -\int_{\bbr^2} \bigg[ \pa_t J^\e_R - \frac{\e^2}{2} \nabla \cdot ( D^\e\psi^\e \otimes \overline{D^\e\psi^\e} + \overline{D^\e\psi^\e} \otimes D^\e\psi^\e) \\
	&\hspace{2.5cm} - \nabla p(|\psi^\e|^2) + \frac{\e^2}{4} \Delta \nabla |\psi^\e|^2 - \frac{\e^{2+2\delta}}{4} \pa_t \pa_t \nabla |\psi^\e|^2 \bigg] \cdot u \, \d x \\
	&= -\frac{\d}{\d t} \int_{\bbr^2} J_R^\e \cdot u \, \d x + \int_{\bbr^2} J_R^\e \cdot \pa_t u \, \d x 
	- \int_{\bbr^2} \frac{\e^2}{2} ( D^\e\psi^\e \otimes \overline{D^\e\psi^\e} + \overline{D^\e\psi^\e} \otimes D^\e\psi^\e ) : \nabla u \, \d x \\
	&\qquad - \int_{\bbr^2} \frac{\e^2}{4} \nabla \rho^\e \cdot (\Delta u) \, \d x - \int_{\bbr^2} (\rho^\e)^\gamma (\nabla \cdot u) \, \d x \\
	&\qquad - \frac{\e^{2+2\delta}}{4} \frac{\d}{\d t} \int_{\bbr^2} \pa_t \rho^\e \nabla \cdot u \, \d x
	+ \frac{\e^{2+2\delta}}{4} \int_{\bbr^2} \pa_t \rho^\e \left( \nabla \cdot \pa_t u \right) \, \d x.
\end{align*}
On the other hand, $ \mathcal{I}_{112} $ can be estimated using the Euler equations \eqref{Euler-CS} as
\begin{align*}
	\mathcal{I}_{112} &= \int_{\bbr^2} J^\e \cdot \left( (u \cdot \nabla) u + \frac{\gamma}{\gamma-1} \nabla (\rho^{\gamma-1}) \right) \, \d x \\
	&= \int_{\bbr^2} J^\e \otimes u : \nabla u \, \d x + \int_{\bbr^2} \frac{\gamma}{\gamma-1} J^\e \cdot \nabla (\rho^{\gamma-1}) \, \d x.
\end{align*}
Combining the estimates for $ \mathcal{I}_{111} $ and $ \mathcal{I}_{112} $, we obtain the estimate for $ \mathcal{I}_{11} $ as
\begin{align}
	\begin{aligned}\label{I_31}
		\mathcal{I}_{11} &= \int_{\bbr^2} \left( J^\e \otimes u \right) : \nabla u \, \d x - \int_{\bbr^2} \frac{\e^2}{2} ( D^\e\psi^\e \otimes \overline{D^\e\psi^\e} + \overline{D^\e\psi^\e} \otimes D^\e\psi^\e ) : \nabla u \, \d x \\		
		&\quad - \int_{\bbr^2} \frac{\e^2}{4} \nabla \rho^\e \cdot (\Delta u) \, \d x - \int_{\bbr^2} (\rho^\e)^\gamma (\nabla \cdot u) \, \d x 
	+ \int_{\bbr^2} \frac{\gamma}{\gamma-1} J^\e \cdot \nabla (\rho^{\gamma-1}) \, \d x \\
		&\quad - \frac{\d}{\d t} \left[ \int_{\bbr^2} J_R^\e \cdot u \, \d x + \int_{\bbr^2} \frac{\e^{2+2\delta}}{4} \pa_t \rho^\e \nabla \cdot u \, \d x \right] \\
		&\quad + \int_{\bbr^2} J^\e_R \cdot \pa_t u \, \d x + \int_{\bbr^2} \frac{\e^{2+2\delta}}{4} \pa_t \rho^\e \left( \nabla \cdot \pa_t u \right) \, \d x.
	\end{aligned}
\end{align}

\vspace{0.2cm}
\noindent$\bullet$ (Estimate of $\mathcal{I}_{12}$): 
Similarly, we decompose $\mathcal{I}_{12}$ as 
\begin{align*}
	\mathcal{I}_{12}
	=\frac{\d}{\d t}\int_{\bbr^2} \frac{1}{2} \rho^\e|u|^2\,\d x = \int_{\bbr^2}\frac{1}{2}(\pa_t \rho^\e)|u|^2\,\d x +\int_{\bbr^2}\rho^\e u \cdot \pa_t u\,\d x =: \mathcal{I}_{121}+\mathcal{I}_{122}.
\end{align*}
To estimate $\mathcal{I}_{121}$, we use the mass conservation law from Proposition \ref{P2.1} (1) to obtain
\[
\mathcal{I}_{121} = \frac{\d}{\d t}\int_{\bbr^2}\frac{1}{2}\rho^\e_R|u|^2\,\d x - \int_{\bbr^2}\rho^\e_R u \cdot\pa_t u\,\d x + \int_{\bbr^2} u\otimes J^\e:\nabla u\,\d x.
\]
On the other hand, using the Euler equations \eqref{Euler-CS}, we estimate $\mathcal{I}_{122}$ as
\begin{align*}
	\mathcal{I}_{122} &= \int_{\bbr^2}\rho^\e u\cdot \left(-(u\cdot\nabla) u-\gamma\rho^{\gamma-2}\nabla\rho\right)\,\d x\\
	&=-\int_{\bbr^2}\left(\rho^\e u\otimes u\right):\nabla u \,\d x - \int_{\bbr^2}\gamma\rho^{\gamma-2} \rho^\e u \cdot\nabla \rho\,\d x.
\end{align*}
Therefore, combining the estimates for $\mathcal{I}_{121}$ and $\mathcal{I}_{122}$, we obtain
\begin{align}
	\begin{aligned}\label{I_32}
		\mathcal{I}_{12} &= \int_{\bbr^2} \left(u\otimes J^\e - \rho^\e u\otimes u\right):\nabla u \,\d x - \int_{\bbr^2} \gamma \rho^\e u \cdot\nabla \rho\,\d x\\
		&\quad+\frac{\d}{\d t}\int_{\bbr^2}\frac{1}{2}\rho^\e_R|u|^2\,\d x - \int_{\bbr^2}\rho^\e_R u\cdot\pa_t u\,\d x.
	\end{aligned}
\end{align}

\vspace{0.2cm}
\noindent $\bullet$ (Estimate of $\mathcal{I}_{13}$): Once again, applying the mass conservation law for $\rho^\e$ and \eqref{Euler-CS}$_1$ from the Euler equations, we obtain
	\begin{align}
	\begin{aligned}\label{I_33}
	\mathcal{I}_{13} &= \frac{\d}{\d t}\int_{\bbr^2}\frac{1}{2^{\gamma}}\rho^\gamma\,\d x -\frac{\d}{\d t}\int_{\bbr^2}\frac{\gamma}{\gamma-1}\rho^\e\rho^{\gamma-1}\,\d x\\
	&=\int_{\bbr^2}\gamma\rho^{\gamma-1}(\pa_t \rho)\,\d x -\int_{\bbr^2}\frac{\gamma}{\gamma-1}\left((\gamma-1)\rho^\e\rho^{\gamma-2}(\pa_t\rho)+\pa_t(\rho^\e)\rho^{\gamma-1}\right)\,\d x\\
	&=\int_{\bbr^2}\gamma\nabla(\rho^{\gamma-1})\cdot(\rho u)\,\d x +\int_{\bbr^2}\gamma\rho^{\gamma-2}\rho^\e\nabla\cdot(\rho u)\,\d x \\
	&\quad-\int_{\bbr^2}\frac{\gamma}{\gamma-1} \pa_t(\rho^\e_R)\rho^{\gamma-1}\,\d x-\int_{\bbr^2}\frac{\gamma}{\gamma-1} J^{\e} \cdot\nabla(\rho^{\gamma-1})\,\d x\\
	&=-\int_{\bbr^2}(\gamma-1)\rho^\gamma\nabla \cdot u\,\d x+\int_{\bbr^2}\gamma\rho^{\gamma-2} \rho^\e u\cdot \nabla\rho \,\d x+\int_{\bbr^2}\gamma\rho^{\gamma-1}\rho^\e\nabla\cdot u\,\d x \\
	&\quad-\int_{\bbr^2}\frac{\gamma}{\gamma-1}J^\e\cdot\nabla(\rho^{\gamma-1}) \,\d x-\frac{\d}{\d t}\int_{\bbr^2}\frac{\gamma}{\gamma-1} \rho^\e_R\rho^{\gamma-1}\,\d x +\int_{\bbr^2}\frac{\gamma}{\gamma-1} \rho_R^\e\pa_t(\rho^{\gamma-1})\,\d x.
	\end{aligned}
	\end{align}
	
\vspace{0.2cm}

\vspace{0.3cm}
Finally, summing \eqref{I_31}--\eqref{I_33} for $\mathcal{I}_{1\ell}$ with $\ell=1,2,3$, we observe that several terms cancel, and we obtain
\begin{align*}
	\frac{\d\mathcal{H}^\e(t)}{\d t}&= \int_{\bbr^2} \bigg[J^\e\otimes u+u\otimes J^\e-\rho^\e u\otimes u -\frac{\e^2}{2} 	(D^\e\psi^\e\otimes\overline{D^\e\psi^\e}+\overline{D^\e\psi^\e}\otimes D^\e\psi^\e)\bigg]:\nabla u\,\d x \\
	&\quad-\int_{\bbr^2}\frac{\e^2}{4}\nabla 
	\rho^\e \cdot(\Delta u) +\left((\rho^\e)^\gamma-\rho^\gamma-\gamma\rho^{\gamma-1}(\rho^\e-\rho)\right)(\nabla\cdot u)\,\d x\\
	&\quad+\int_{\bbr^2} J^\e_R\cdot \pa_t u+\frac{\e^{2+2\delta}}{4}\pa_t\rho^\e  \left(\nabla \cdot \pa_t u\right)-\rho^\e_R u\cdot\pa_t u+\frac{\gamma}{\gamma-1}\rho^\e_R\pa_t(\rho^{\gamma-1})\,\d x\\
	&\quad-\frac{\d}{\d t}\left[\int_{\bbr^2} J^\e_{R}\cdot u +\frac{\e^{2+2\delta}}{4}\pa_t\rho^\e \nabla \cdot u-\frac12\rho^\e_R|u|^2+\frac{\gamma}{\gamma-1}\rho^\e_R\rho^{\gamma-1}\,\d x \right].
\end{align*} 
Since the correction functional $\mathcal{R}^\e$ defined in \eqref{correction} exactly corresponds to the last term in the above equation, the proof is completed.

\end{proof}

We now present the proof of Proposition \ref{P3.1}.

\begin{proof}[Proof of Proposition \ref{P3.1}]
 It suffices to show that the right-hand side of \eqref{modulated-energy-est} can be bounded by $ C\mathcal{H}^\e+ C\e^{\min\{1,\delta\}} $. To this end, we decompose the right-hand side of \eqref{modulated-energy-est} into the sum of seven terms $\mathcal{I}_{2\ell}, \ell=1,2,\ldots,7 $. Each term is defined as follows:
\begin{align*}
	\mathcal{I}_{21}&:=-\int_{\bbr^2} \left[ \frac{\e^2}{2} (D^\e\psi^\e\otimes\overline{D^\e\psi^\e}+\overline{D^\e\psi^\e}\otimes D^\e\psi^\e)-J^\e\otimes u-u\otimes J^\e+\rho^\e u\otimes u\right]:\nabla u\,\d x, \\	
	\mathcal{I}_{22}&:= -\int_{\bbr^2}\frac{\e^2}{4}\nabla \rho^\e \cdot(\Delta u)\,\d x, \qquad
	\mathcal{I}_{23}:=-\int_{\bbr^2}\left((\rho^\e)^\gamma-\rho^\gamma-\gamma\rho^{\gamma-1}(\rho^\e-\rho)\right)(\nabla\cdot u)\,\d x,\\
	\mathcal{I}_{24}&:=\int_{\bbr^2} J_R^\e\cdot \pa_t u\,\d x, \qquad
	\mathcal{I}_{25}:=\int_{\bbr^2}\frac{\e^{2+2\delta}}{4}\pa_t\rho^\e(\nabla\cdot\pa_t u)\,\d x,\\
	\mathcal{I}_{26}&:=-\int_{\bbr^2}\rho^\e_R u\cdot\pa_t u\,\d x, \qquad
	\mathcal{I}_{27}:=\frac{\gamma}{\gamma-1}\int_{\bbr^2}\rho^\e_R\pa_t(\rho^{\gamma-1})\,\d x.
\end{align*}

In what follows, we estimate each term $\mathcal{I}_{2\ell}$ separately.

\noindent $\bullet$ (Estimate of $\mathcal{I}_{21}$):
First, we observe that
\begin{align*}
	\frac{1}{2}&\bigl((\e D^{\e}\psi^\e-\i u\psi^\e)\otimes\overline{\e D^{\e}\psi^\e-\i u\psi^\e}
	+\overline{\e D^{\e}\psi^\e-\i u\psi^\e}\otimes (\e D^{\e}\psi^\e-\i u\psi^\e)\bigr)\\
	&=\frac{\e^2}{2}\bigl(D^\e\psi^\e\otimes\overline{D^\e\psi^\e}+\overline{D^\e\psi^\e}\otimes D^\e\psi^\e\bigr)
	-J^\e\otimes u - u\otimes J^\e + \rho^\e u\otimes u.
\end{align*}
Hence,
\begin{align*}
	\mathcal{I}_{21}
	&= -\frac{1}{2}\int_{\bbr^2}\bigl((\e D^{\e}\psi^\e-\i u\psi^\e)\otimes\overline{\e D^{\e}\psi^\e-\i u\psi^\e}
	+\overline{\e D^{\e}\psi^\e-\i u\psi^\e}\otimes (\e D^{\e}\psi^\e-\i u\psi^\e)\bigr):\nabla u \,\d x\\
	&\le C\int_{\bbr^2}|(\e D^\e-\i u)\psi^\e|^2\,\d x
	\le C\mathcal{H}^\e,
\end{align*}
where we used the smoothness of $ u $ (in particular, the boundedness of $\|\nabla u\|_{L^\infty}$) in the last inequality.

\medskip

\noindent $\bullet$ (Estimate of $\mathcal{I}_{22}$): 
Since $\nabla \rho^\e=2\Re(\overline{\psi^\e}D^\e\psi^\e)$, we use the energy conservation in Proposition~\ref{P2.1}\,(3) together with the smoothness of $u$ to get
\[
\mathcal{I}_{22}
= -\int_{\bbr^2}\frac{\e^2}{2}\Re(\overline{\psi^\e}D^\e\psi^\e)\cdot(\Delta u)\,\d x
\le \frac{\e}{2} \|\e D^\e\psi^\e\|_{L^2}\|\psi^\e\|_{L^{2\gamma}}\|\Delta u\|_{L^{\frac{2\gamma}{\gamma-1}}}
\le C\e.
\]

\medskip

\noindent $\bullet$ (Estimate of $\mathcal{I}_{23}$): 
Since 
\[
p(\rho^\e|\rho)=(\rho^\e)^\gamma-\rho^\gamma-\gamma\rho^{\gamma-1}(\rho^\e-\rho)
\]
is nonnegative for any $\rho^\e$ and $\rho$, we have
\[
\mathcal{I}_{23}
=-\int_{\bbr^2}\bigl((\rho^\e)^\gamma-\rho^\gamma-\gamma\rho^{\gamma-1}(\rho^\e-\rho)\bigr)(\nabla\cdot u)\,\d x
\le C\int_{\bbr^2}p(\rho^\e|\rho)\,\d x
\le C\mathcal{H}^\e.
\]

\medskip

\noindent $\bullet$ (Estimate of $\mathcal{I}_{24}$): 
We use the definition 
\[
J^\e_R=\e^{2+2\delta}\Re(\overline{D^\e_t\psi^\e} D^{\e}\psi^\e )
\]
and the total energy conservation to obtain
\[
\int_{\bbr^2}|J^\e_R|\,\d x
\le \int_{\bbr^2}\e^{\delta}|\e^{1+\delta}D^\e_t\psi^\e||\e D^\e\psi^\e|\,\d x 
\le \e^\delta \|\e^{1+\delta} D^\e_t\psi^\e\|_{L^2} \|\e D^\e\psi^\e\|_{L^2}
\le  C\e^\delta.
\]
Hence,
\begin{equation}\label{I_46}
	\mathcal{I}_{24}
	\le \int_{\bbr^2}|J^\e_R||\pa_t u|\,\d x
	\le  C\e^\delta.
\end{equation}

\medskip

\noindent $\bullet$ (Estimate of $\mathcal{I}_{25}$): 
Noting that $\pa_t \rho^\e = 2\Re(\overline{\psi^\e} D^\e_t\psi^\e)$, we estimate $\mathcal{I}_{25}$ as 
\begin{align*}
	\mathcal{I}_{25}
	=\int_{\bbr^2}\frac{\e^{2+2\delta}}{2}\Re(\overline{\psi^\e} D^\e_t\psi^\e)(\nabla\cdot\pa_t u)\,\d x
	\le C\int_{\bbr^2}\e^{1+\delta}|\psi^\e||\e^{1+\delta} D^\e_t\psi^\e||\nabla\cdot\pa_t u|\,\d x.
\end{align*}
Using again the conservation of total energy, we obtain
\begin{equation}\label{I_47}
	\mathcal{I}_{25} 
	\le C\e^{1+\delta}\|\psi^\e\|_{L^{2\gamma}}\|\e^{1+\delta} D_t^\e\psi^\e\|_{L^2}\|\nabla\cdot\pa_t u\|_{L^\frac{2\gamma}{\gamma-1}}
	\le C\e^{1+\delta}.
\end{equation}
\medskip
\noindent $\bullet$ (Estimates of $\mathcal{I}_{26}$ and $\mathcal{I}_{27}$): 
Since $\rho^\e_R := \e^{1+2\delta}\Im(\overline{\psi^\e}D_t^\e\psi^\e)$ and $\rho, u$ are smooth, we estimate $\mathcal{I}_{27}$ as
\begin{align}
	\begin{aligned}\label{I_48}
		\mathcal{I}_{27}
		&\le C\int_{\bbr^2} |\rho^\e_R||\pa_t(\rho^{\gamma-1})|\,\d x
		\le C\e^{\delta}\int_{\bbr^2}|\psi^\e||\e^{1+\delta} D^\e_t\psi^\e||\pa_t(\rho^{\gamma-1})|\,\d x\\
		&\le C\e^{\delta}\|\psi^\e\|_{L^{2\gamma}}\|\e^{1+\delta} D^\e_t\psi^\e\|_{L^2}\|\pa_t\rho^{\gamma-1}\|_{L^{\frac{2\gamma}{\gamma-1}}}
		\le C\e^\delta.
	\end{aligned}
\end{align}
A similar argument yields
\[
\mathcal{I}_{26}\le C\e^\delta.
\]

Finally, combining all the above estimates for $\mathcal{I}_{2\ell}$, we conclude that
\[
\frac{\d}{\d t}\bigl(\mathcal{H}^\e(t)+\mathcal{R}^\e(t)\bigr)
\le C\mathcal{H}^\e(t) +C\e^{\min\{1,\delta\}},
\]
which completes the proof of Proposition~\ref{P3.1}.
\end{proof}

Our final goal in this section is to establish the following modulated energy estimate.

\begin{proposition}
	Assume $\gamma>1$. Let $(\psi^\e, A^\e_0, A^\e)$ be the global solution to the CSH system \eqref{CSH_parameter}, 
	and let $(\rho,u,A_0, A)$ be the local smooth solution to the Euler--CS system \eqref{Euler-CS} on $[0,T_*)$, 
	subject to the initial data $(\psi^\e_{\textup{in}},A^{\e}_{0,\textup{in}},A^{\e}_{\textup{in}})$ 
	and $(\rho_{\textup{in}},u_{\textup{in}},A_{0,\textup{in}},A_{\textup{in}})$ 
	satisfying assumptions $(\mathcal{C}1)$ and $(\mathcal{C}2)$, respectively. Then, for $0\leq t\leq T_*$, we have
	\[
	\mathcal{H}^\e(t)\le C\e^{\alpha},\quad\mbox{where}\quad \alpha = \min\{1,\delta,\lambda\}.
	\]
\end{proposition}

\begin{proof}
	We integrate \eqref{modulated_auxiliary_est} over $[0,t]$ with $t\le T_*$ to obtain
	\[
	\mathcal{H}^\e(t)
	\le \mathcal{H}^\e(0) - \bigl(\mathcal{R}^\e(t)-\mathcal{R}^\e(0)\bigr)
	+ C\int_0^t \mathcal{H}^\e(s)\,\d s
	+ C\e^{\min\{1,\delta\}}.
	\]
Next, recall that the relativistic correction functional $\mathcal{R}^\e$ defined in \eqref{correction} is given by
\[
\mathcal{R}^\e(t)
= \int_{\bbr^2} J^\e_R \cdot u \,\d x
+ \int_{\bbr^2}\frac{\e^{2+2\delta}}{4}\pa_t \rho^\e \nabla \cdot u \,\d x
- \int_{\bbr^2}\frac{1}{2}\rho^\e_R|u|^2\,\d x
+ \frac{\gamma}{\gamma-1}\int_{\bbr^2}\rho^\e_R\rho^{\gamma-1}\,\d x.
\]
Note that each term in $\mathcal{R}^\e$ can be bounded by $C\e^{\delta}$, as shown in \eqref{I_46}, \eqref{I_47}, and \eqref{I_48}. For instance, the first term satisfies

\[
\Bigl|\int_{\bbr^2}J^\e_R\cdot u\,\d x\Bigr|
\le \int_{\bbr^2}|J^\e_R||u|\,\d x
\le C\e^\delta,
\]
and the remaining terms can be treated similarly. Hence,
\[\left|\mathcal{R}^\e(t)\right|\le C\e^\delta,\qquad 0\leq t\le T_*.\]
Therefore,
\[
\mathcal{H}^\e(t)
\le \mathcal{H}^\e(0)
+ C\int_0^t \mathcal{H}^\e(s)\,\d s
+ C\e^{\min\{1,\delta\}},\quad 0\le t\le T_*.
\]
Finally, using Gr\"onwall's inequality and the well-prepared initial data condition $(\mathcal{C}1)$, we obtain
\[
\mathcal{H}^\e(t)\le C\e^{\min\{1,\delta,\lambda\}},\quad 0\le t\le T_*,
\]
which is the desired estimate.

\end{proof}

\section{Quantitative hydrodynamic limits of the CSH system}\label{sec:4}
\setcounter{equation}{0}

In this section, we derive the quantitative hydrodynamic limits of the CSH system based on the modulated energy estimates established in the previous section,  thereby completing the proof of Theorem~\ref{Thm:main}. We begin by citing, without proof, a technical lemma that will be used to establish the convergence of the density.

\begin{lemma}\label{L4.1}\cite{CJ21,KM22,LM98}
	Let $\gamma > 1$, and let $\rho$ be a function on $\mathbb{R}^{1+2}$. 	Define the relative pressure functional by
	\[
	p(n|\rho) = n^\gamma - \rho^\gamma - \gamma \rho^{\gamma-1}(n- \rho).
	\]
Then, the following statements hold.
	\begin{enumerate}
		\item  If $ \gamma \ge 2 $, then
		\[|n - \rho|^\gamma \le p(n|\rho).\]
		\item If there exist two positive constants $\underline{\rho}$ and $\overline{\rho}$ such that
			\[
		0 < \underline{\rho} < \rho(t,x) < \overline{\rho}, \qquad \text{for all }\,\, (t,x) \in [0,T_*) \times K,
		\]
		for some $K\subset \R^2$, then
		\begin{align*}
			p(n|\rho)
			&\geq \gamma(\gamma - 1)\min\{n^{\gamma - 2},\, \rho^{\gamma - 2}\}(n - \rho)^2\\
			&\geq C
			\begin{cases}
				(n - \rho)^2,& \text{if } \,\,\dfrac{\rho}{2} \le n \le 2\rho,\\
				1 + n^\gamma, & \text{otherwise},
			\end{cases}
		\end{align*}
		where $ C = C(\underline{\rho}, \overline{\rho}, \gamma) > 0 $. 
	\end{enumerate}
\end{lemma}

\
\\
We also note that the first term in the modulated energy $\mathcal{H}^\e$ can be expressed in terms of hydrodynamic quantities as
\begin{align*}
	\int_{\bbr^2}\frac{1}{2}|(\e D^\e - \i u)\psi^\e|^2\,\d x 
	= \int_{\bbr^2}\frac{1}{2}\rho^\e|u^\e - u|^2\,\d x
	+ \int_{\bbr^2}\frac{\e^2}{2}|\nabla\sqrt{\rho^\e}|^2\,\d x,
\end{align*}
which implies
\begin{align}\label{4.1}
	\|\sqrt{\rho^\e}\,|u^\e - u|\|_{L^2(\bbr^2)} \le (\mathcal{H}^\e)^{\frac12}.
\end{align}
This observation will be used to handle the convergence of the momentum.

\begin{proof}[Proof of Theorem~\ref{Thm:main}]

We split the proof into two cases depending on the range of $\gamma$.

\medskip
\noindent $\bullet$ (Case of $\gamma \geq 2$): We first consider the case where $\gamma \ge 2$.

\medskip
\noindent $\diamond$ (Convergence of $\rho^\e$): By Lemma~\ref{L4.1}, we obtain
\[
\int_{\bbr^2} |\rho^\e - \rho|^\gamma\,\d x \le \int_{\bbr^2} p(\rho^\e|\rho)\,\d x \le C\mathcal{H}^\e \le C\e^\alpha,
\]
which yields
\[
\|\rho^\e - \rho\|_{L^\gamma(\bbr^2)} \le C\e^{\frac{\alpha}{\gamma}}.
\]

\medskip
\noindent $\diamond$ (Convergence of $J^\e$ and $\sqrt{\rho^\e}u^\e$): For $p = \frac{2\gamma}{\gamma+1}$, we apply H\"older's inequality to estimate
\begin{align*}
	\|J^\e - \rho u\|_{L^p(\bbr^2)} &\le \|\rho^\e(u^\e - u)\|_{L^p(\bbr^2)} + \|(\rho^\e - \rho)u\|_{L^p(\bbr^2)} \\
	&\le \|\sqrt{\rho^\e}\|_{L^{2\gamma}(\bbr^2)} \|\sqrt{\rho^\e}|u^\e - u|\|_{L^2(\bbr^2)} + \|\rho^\e - \rho\|_{L^\gamma(\bbr^2)} \|u\|_{L^{\frac{2\gamma}{\gamma-1}}(\bbr^2)} \\
	&\le C\|\sqrt{\rho^\e}|u^\e - u|\|_{L^2(\bbr^2)} + C\|\rho^\e - \rho\|_{L^\gamma(\bbr^2)} \\
	&\le C(\mathcal{H}^\e)^{\frac12} + C\e^{\frac{\alpha}{\gamma}} \\
	&\le C\e^{\frac{\alpha}{2}} + C\e^{\frac{\alpha}{\gamma}} \le C\e^{\frac{\alpha}{\gamma}},
\end{align*}
where we used the bound from \eqref{4.1}, the smoothness of $ u $, and the boundedness of $ \rho^\e $ in $ L^\gamma(\bbr^2) $.

Similarly, we estimate
\begin{align*}
	\|\sqrt{\rho^\e}u^\e-\sqrt{\rho}u\|_{L^2(\bbr^2)} &\le \|\sqrt{\rho^\e}|u^\e-u|\|_{L^2(\bbr^2)} +\|(\sqrt{\rho^\e}-\sqrt{\rho})u\|_{L^2(\bbr^2)} \\
	&\le C(\mathcal{H}^\e)^{\frac12} + \|\sqrt{\rho^\e}-\sqrt{\rho}\|_{L^{2\gamma}(\bbr^2)}\|u\|_{L^{\frac{2\gamma}{\gamma-1}}(\bbr^2)} \\
	&\le C(\mathcal{H}^\e)^{\frac12} + C\|\rho^\e - \rho\|_{L^\gamma(\bbr^2)}^{\frac12} \le C\e^{\frac{\alpha}{2\gamma}}.
\end{align*}

\medskip
\noindent $\diamond$ (Vanishing of $\rho^\e_R$ and $J^\e_R$): We first show that $J^\e_R$ vanishes in $L^1(\bbr^2)$ by the estimate
\[
\int_{\bbr^2}|J^\e_R|\,\d x \le \e^\delta \|\e^{1+\delta} D^\e_t\psi^\e\|_{L^2(\bbr^2)} \|\e D^\e\psi^\e\|_{L^2(\bbr^2)} \le C\e^\delta.
\]
For $\rho^\e_R$, we note that for any $1 < p < 2$, it holds that
\begin{align*}
	\int_{\bbr^2}|\rho^\e_R|^p\,\d x 
	&= \int_{\bbr^2} \e^{\delta p} |\psi^\e|^p |\e^{1+\delta} D^\e_t\psi^\e|^p\,\d x \\
	&\le \e^{\delta p} \|\rho^\e\|_{L^{\frac{p}{2-p}}(\bbr^2)}^{\frac{p}{2}} \|\e^{1+\delta} D^\e_t\psi^\e\|_{L^2(\bbr^2)}^p \\
	&\le C\e^{\delta p} \|\rho^\e\|_{L^{\frac{p}{2-p}}(\bbr^2)}^{\frac{p}{2}}.
\end{align*}
Choosing $p = \frac{2\gamma}{\gamma+1}$, we obtain
\[
\|\rho^\e_R\|_{L^{\frac{2\gamma}{\gamma+1}}(\bbr^2)} \le C\e^\delta \|\rho^\e\|_{L^\gamma(\bbr^2)}^{\frac12} \le C\e^\delta.
\]

\medskip
\noindent $\diamond$ (Convergence of $A^{\e}_0$ and $\e^{\delta} A^\e$): 
Recall that the gauge equations read
\[
\pa_t (\e^{\delta}A^{\e}) - \nabla A^\e_0 = (\rho^\e u^\e)^{\perp}, 
\quad 
\pa_t A - \nabla A_0 = (\rho u)^{\perp},
\]
which yield
\[
-\Delta A_0^\e = \nabla\cdot(\rho^\e u^\e)^{\perp}, 
\quad 
-\Delta A_0 = \nabla\cdot(\rho u)^{\perp},
\]
under the Coulomb gauge condition $\nabla \cdot A^\e = 0$ and $\nabla \cdot A = 0$.
Using the Hardy--Littlewood--Sobolev inequality, we obtain
\begin{align*}
	\|A_0^\e - A_0\|_{L^{2\gamma}(\R^2)} 
	&\le C\|\rho^\e u^\e - \rho u\|_{L^{\frac{2\gamma}{\gamma+1}}(\R^2)} 
	\le C\e^{\frac{\alpha}{\gamma}}, \\
	\|\nabla(A_0^\e - A_0)\|_{L^{\frac{2\gamma}{\gamma+1}}(\R^2)} 
	&\le C\|\rho^\e u^\e - \rho u\|_{L^{\frac{2\gamma}{\gamma+1}}(\R^2)}
	\le C\e^{\frac{\alpha}{\gamma}}.
\end{align*}
For the convergence of $\e^{\delta} A^\e$, we note that 
$\e^\delta A^\e - A$ satisfies
\[
\pa_t(\e^\delta A^\e - A) 
- \nabla(A^\e_0 - A_0) 
= (\rho^\e u^\e)^{\perp} - (\rho u)^{\perp},
\]
which implies
\begin{align*}
	\|\e^\delta A^\e - A\|_{L^{\frac{2\gamma}{\gamma+1}}(\R^2)}
	&\le \|\e^\delta A^\e_{\textup{in}} - A_{\textup{in}}\|_{L^{\frac{2\gamma}{\gamma+1}}(\R^2)} \\
	&\quad + \int_0^t \Bigl(\|\nabla(A_0^\e - A_0)\|_{L^{\frac{2\gamma}{\gamma+1}}(\R^2)} 
	+ \|\rho^\e u^\e - \rho u\|_{L^{\frac{2\gamma}{\gamma+1}}(\R^2)}\Bigr)\,\d\tau \\
	&\le C\e^\alpha + C T_*\,\e^{\frac{\alpha}{\gamma}} \le C\e^{\frac{\alpha}{\gamma}}.
\end{align*}

\medskip
To sum up, we obtain the following quantitative hydrodynamic limit estimates for the CSH system when $\gamma \ge 2$:
\begin{align*}
	&\|\rho^\e - \rho\|_{L^\gamma(\bbr^2)} \le C\e^{\frac{\alpha}{\gamma}},\quad
	\|J^\e - \rho u\|_{L^{\frac{2\gamma}{\gamma+1}}(\R^2)} \le C\e^{\frac{\alpha}{\gamma}},\quad
	\|\sqrt{\rho^\e} u^\e - \sqrt{\rho}u\|_{L^2(\bbr^2)} \le C\e^{\frac{\alpha}{2\gamma}},\\
	&\|J^\e_R\|_{L^1(\bbr^2)} \le C\e^\delta,\quad
	\|\rho^\e_R\|_{L^{\frac{2\gamma}{\gamma+1}}(\bbr^2)} \le C\e^\delta,\\
	&	\|A_0^\e - A_0\|_{L^{2\gamma}(\R^2)} \le C\e^{\frac{\alpha}{\gamma}}, \quad
	\|\nabla A_0^\e - \nabla A_0 \|_{L^{\frac{2\gamma}{\gamma+1}}(\R^2)}
	\le C\e^{\frac{\alpha}{\gamma}},\quad	\|\e^\delta A^\e - A\|_{L^{\frac{2\gamma}{\gamma+1}}(\R^2)}\le C\e^{\frac{\alpha}{\gamma}}.
\end{align*}
This verifies the first part of Theorem~\ref{Thm:main}.

\vspace{0.3cm}

\noindent $\bullet$ (Case of $1<\gamma<2$): Let $K \subset \mathbb{R}^2$ be any compact subset. Since $\rho$ is smooth on $[0,T_*)\times K$, there exist positive  $\underline{\rho}$ and $\overline{\rho}$ such that
\[
0 < \underline{\rho} \le \rho(t,x) \le \overline{\rho},
\quad \text{for all } 0 \le t < T_*,\; x\in K.
\]
We begin by splitting the local $L^\gamma$-norm of $\rho^\varepsilon - \rho$ over $K$ into two parts:
\begin{align*}
	\int_{K}|\rho^\e - \rho|^\gamma\,\d x
	= \int_{K \cap \{\tfrac{\rho}{2}\le \rho^\e \le 2\rho\}}|\rho^\e - \rho|^\gamma\,\d x
	+ \int_{K \cap \{\tfrac{\rho}{2}\le \rho^\e \le 2\rho\}^c}|\rho^\e - \rho|^\gamma\,\d x
	=: \mathcal{I}_{31} + \mathcal{I}_{32}.
\end{align*}
To estimate $\mathcal{I}_{31}$, we note that on the set
$\{\tfrac{\rho}{2}\le \rho^\varepsilon \le 2\rho\}$, both $\rho$ and $\rho^\varepsilon$ are comparably bounded. By Lemma~\ref{L4.1} and Hölder’s inequality, we obtain
\begin{align*}
	\mathcal{I}_{31}&=\int_{K\cap \{\tfrac{\rho}{2}\leq \rho^\e\leq 2\rho\}}	\min\bigl\{(\rho^\e)^{\frac{\gamma(\gamma-2)}{2}},\rho^{\frac{\gamma(\gamma-2)}{2}}\bigr\}|\rho^\e-\rho|^{\gamma}\max\bigl\{(\rho^\e)^{\frac{\gamma(2-\gamma)}{2}},\rho^{\frac{\gamma(2-\gamma)}{2}}\bigr\}\,\d x\\
	&\le \Bigl(\int_{K \cap \{\tfrac{\rho}{2}\le \rho^\e \le 2\rho\}}
	\min\{(\rho^\e)^{\gamma-2},\rho^{\gamma-2}\}\,|\rho^\e - \rho|^2\,\d x\Bigr)^{\frac{\gamma}{2}}
	\Bigl(\int_{K \cap \{\tfrac{\rho}{2}\le \rho^\e \le 2\rho\}}
	\max\{(\rho^\e)^\gamma,\rho^\gamma\}\,\d x\Bigr)^{\frac{2-\gamma}{2}}\\
	&\le C\Bigl(\int_{\bbr^2} p(\rho^\e|\rho)\,\d x\Bigr)^{\frac{\gamma}{2}}
	\;\le\; C(\mathcal{H}^\e)^{\frac{\gamma}{2}}.
\end{align*}
Next, we again invoke Lemma~\ref{L4.1} and observe that on the complement 
$\{\tfrac{\rho}{2}\le \rho^\varepsilon \le 2\rho\}^c$, either 
$\rho^\varepsilon < \tfrac{\rho}{2}$ or $\rho^\varepsilon > 2\rho$, to estimate
\begin{align*}
	\mathcal{I}_{32}
	\le C \int_{K \cap \{\tfrac{\rho}{2}\le\rho^\e\le2\rho\}^c} \bigl(1 + (\rho^\e)^\gamma\bigr)\,\d x
	&\le C \int_{\bbr^2} p(\rho^\e|\rho)\,\d x
	\;\le\; C\mathcal{H}^\e.
\end{align*}

Combining the estimates of $\mathcal{I}_{31}$ and $\mathcal{I}_{32}$, we have
\[
\int_{K}|\rho^\e-\rho|^\gamma\,\d x
\;\le\;
C(\mathcal{H}^\e)^{\frac{\gamma}{2}} + C\mathcal{H}^\e
\;\le\;
C(\mathcal{H}^\e)^{\frac{\gamma}{2}}
\;\le\;
C\e^{\frac{\alpha\gamma}{2}}.
\]
Hence, for $1<\gamma<2$, we obtain
\[
\|\rho^\e - \rho\|_{L^\gamma(K)} \;\le\; C\,\e^{\frac{\alpha}{2}}
\quad \text{for any compact subset } K \subset \bbr^2,
\]
i.e., the convergence of the density holds locally. The remaining estimates for $J^\e$, $\sqrt{\rho^\e}u^\e$, $\rho^\e_R$, and $J^\e_R$ can be obtained in the same way as in the case $\gamma \ge 2$, except that the corresponding bounds are now local whenever they rely on the density convergence.
 To be more specific, we have the following convergences: for any compact subset $K$ of $\bbr^2$,
\begin{align*}
	&\|J^\e-\rho u\|_{L^{\frac{2\gamma}{\gamma+1}}(K)}\le C\e^{\frac{\alpha}{2}},\quad \|\sqrt{\rho^\e}u^\e-\sqrt{\rho}u\|_{L^2(K)}\le C\e^{\frac{\alpha}{2}},\\
	&\|J^\e_R\|_{L^1(\bbr^2)}\le C\e^\delta,\quad \|\rho^\e_R\|_{L^\frac{2\gamma}{\gamma+1}(\bbr^2)}\le C\e^{\delta},\\
	&\|A_0^\e - A_0\|_{L^{2\gamma}(K)} \le C\e^{\frac{\alpha}{2}}, \quad
	\|\nabla A_0^\e - \nabla A_0 \|_{L^{\frac{2\gamma}{\gamma+1}}(K)}
	\le C\e^{\frac{\alpha}{2}},\quad	\|\e^\delta A^\e - A\|_{L^{\frac{2\gamma}{\gamma+1}}(K)}\le C\e^{\frac{\alpha}{2}},
\end{align*}
which completes the proof of the second part of Theorem \ref{Thm:main}. \end{proof}

\section{Conclusion}\label{sec:5}

	In this work, we have investigated the simultaneous non-relativistic and semi-classical limit of the CSH system and rigorously justified its convergence toward the Euler--CS system with explicit rates. 
	Our result provides a direct hydrodynamic limit from the relativistic CSH system to the Euler--CS system through a single scaling, thereby unifying the previously studied non-relativistic limit from CSH to CSS and the semi-classical limit from CSS to Euler--CS. 
	The analysis is based on a modulated energy framework, which allows us to quantitatively measure the distance between the modulated CSH system and its hydrodynamic limit. 
	This approach yields stability estimates leading to the convergence of density, momentum, and gauge fields, as well as the vanishing of relativistic correction terms in the limit.
	
	Several natural directions for further investigation arise from the present work. 
	In particular, it would be of interest to explore the hydrodynamic limit in the absence of the self-interaction potential $V$, or under different interaction potentials. 
	In the present analysis, the nonlinear potential is closely related to the pressure structure and plays an important role in the convergence of the density. 
	Understanding whether analogous hydrodynamic limits can be obtained without relying on this structure, especially in the relativistic setting, remains an interesting direction for future research. 
	Another possible direction concerns semi-classical limits toward relativistic hydrodynamic models. 
	While the present work focuses on convergence toward the classical Euler--CS system, it is natural to ask whether suitable scalings and appropriate reformulations may lead, in the semi-classical regime, to relativistic quantum hydrodynamic systems and their corresponding classical limits. 
	These problems are left for future work.

\appendix

\section{Proof of Proposition \ref{P2.1}}\label{Appen:A}
In this section, we provide a detailed proof of Proposition \ref{P2.1}.

\begin{proof}
	We begin by recalling the CSH system \eqref{CSH_parameter}:
	\begin{align}\label{A.1}
		\begin{aligned}
			&\i \e D^\e_t\psi^\e-\frac{\e^{2+2\delta}}{2}D^\e_tD^\e_t \psi^\e +\frac{\e^2}{2} \left(D^\e_1D^\e_1\psi^\e+D^\e_2D^\e_2\psi^\e\right) - V'\left(|\psi^\e|^2\right)\psi^\e=0,\\
			&\e^{\delta}\partial_tA^\e_1-\partial_1 A^\e_0=-\e\Im(\overline{\psi^\e}D^\e_2\psi^\e),\quad
			\e^{\delta}\partial_tA^\e_2-\partial_2 A^\e_0=\e\Im(\overline{\psi^\e}D^\e_1\psi^\e),\\
			&\e^{\delta}(\partial_1A^\e_2-\partial_2 A^\e_1)=-|\psi^\e|^2 + \e^{1+2\delta}\Im\left(\overline{\psi^\e} D^\e_t\psi^\e\right).
		\end{aligned}
	\end{align} 
	
	\noindent $\bullet$ (Mass conservation):  
	To derive the mass conservation law, we multiply $\eqref{A.1}_1$ by $\overline{\psi^\e}$ and take the imaginary part, yielding  
	\begin{equation*}
		\e\pa_t|\psi^\e|^2- \e^{2+2\delta}\Im(\overline{\psi^\e}D^\e_tD^\e_t\psi^\e) + \e^2\Im\left(\overline{\psi^\e} (D^\e_1D^\e_1\psi^\e+D^\e_2D^\e_2\psi^\e) \right)=0.
	\end{equation*}
	To rewrite the above expression in divergence form, we use the identities
	\begin{align*}
		\pa_t (\overline{\psi^\e}D^\e_t\psi^\e) &= \overline{D^\e_t\psi^\e}D^\e_t \psi^\e+\overline{\psi^\e} D^\e_tD^\e_t \psi^\e,\\
		\pa_j (\overline{\psi^\e}D^\e_j\psi^\e) &= \overline{D^\e_j\psi^\e}D^\e_j \psi^\e+\overline{\psi^\e} D^\e_jD^\e_j \psi^\e,\qquad\mbox{for}\quad j=1,2.
	\end{align*}
	Applying these relations, we obtain:
	\[
	\pa_t\left(|\psi^\e|^2-\e^{1+2\delta}\Im(\overline{\psi^\e}D^\e_t\psi^\e)\right) +\nabla\cdot\left(\e\Im\left(\overline{\psi^\e}D^\e\psi^\e\right)\right)=0.
	\]
This proves the mass conservation law.

\vspace{0.2cm}
\noindent $\bullet$ (Momentum conservation): We multiply $\eqref{A.1}_1$ by $ D^\e\psi^\e $ and take the real part, which gives
\begin{align}\label{A.2.1}
	\begin{aligned}
		&\e\Im(\overline{D^\e_t\psi^\e}D^\e\psi^\e)-\frac{\e^{2+2\delta}}{2}\Re(\overline{D^\e_tD^\e_t \psi^\e} D^\e\psi^\e)\\
		&\quad+\frac{\e^2}{2}\Re\left((\overline{D^\e_1D^\e_1\psi^\e}+\overline{D^\e_2D^\e_2\psi^\e}){D^\e \psi^\e}\right)-V'\left(|\psi^\e|^2\right)\Re(\overline{\psi^\e} D^\e\psi^\e)=0.
	\end{aligned}
\end{align} 

Next, applying $ D^\e $ to $\eqref{A.1}_1$ and multiplying by $ \overline{\psi^\e} $, we obtain
\begin{align}\label{A.2}
	\begin{aligned}
		&\i\e\overline{\psi^\e}D^\e D^\e_t\psi^\e -\frac{\e^{2+2\delta}}{2}\overline{\psi^\e}D^\e D^\e_tD^\e_t \psi^\e   +\frac{\e^2}{2} \overline{\psi^\e} \left(D^\e D^\e_1D^\e_1\psi^\e+D^\e D^\e_2D^\e_2\psi^\e\right)\\
		&\hspace{2.5cm}-\nabla\left( V'\left(|\psi^\e|^2\right)\right)|\psi^\e|^2-V'\left(|\psi^\e|^2\right)\overline{\psi^\e}D^\e\psi^\e =0.
	\end{aligned}
\end{align} 
To reorder the covariant derivatives, we use the identity:
\begin{align*}
	D^\e D^\e_t\psi^\e - D^\e_t D^\e \psi^\e = \frac{\i}{\e}\psi^\e(\e^{\delta}\pa_tA^\e - \nabla A^\e_0),
\end{align*}
which leads to

\begin{align*}
	D^\e D^\e_t D^\e_t \psi^\e
	&= D^\e_t D^\e_t D^\e \psi^\e + \frac{2\i}{\e}D^\e_t\psi^\e(\e^{\delta}\pa_tA^\e - \nabla A^\e_0) 
	+ \frac{\i}{\e}\psi^\e\pa_t(\e^{\delta}\pa_t A^\e - \nabla A^\e_0).
\end{align*}    
Similarly, the spatial covariant derivatives satisfy the following identity for $ j=1,2 $:
\begin{align*}
	D^\e D^\e_j D^\e_j \psi^\e &= D^\e_j D^\e D^\e_j \psi^\e + \frac{\i}{\e^{1-\delta}}D^\e_j\psi^\e(\pa_jA^\e - \nabla A^\e_j)\\
	&= D^\e_j D^\e_j D^\e \psi^\e + \frac{2\i}{\e^{1-\delta}}D^\e_j\psi^\e(\pa_jA^\e - \nabla A^\e_j) 
	+ \frac{\i}{\e^{1-\delta}}\psi^\e\pa_j(\pa_jA^\e - \nabla A^\e_j).
\end{align*}
Using these identities, we take the real part of \eqref{A.2} and express each term as follows:
	\begin{align*}
		I_1&:= -\e\Im (\overline{\psi^\e} D^\e D^\e_t \psi^\e)= -\e\Im (\overline{\psi^\e}D^\e_tD^\e\psi^\e) - |\psi^\e|^2(\e^{\delta}\pa_t A^\e-\nabla A_0^\e)\\
		&=-\e\pa_t \Im(\overline{\psi^\e} D^\e \psi^\e) + \e\Im(\overline{D_t^\e\psi^\e} D^\e \psi^\e) - |\psi^\e|^2(\e^{\delta}\pa_t A^\e-\nabla A_0^\e),\\
					I_2&:=-\frac{\e^{2+2\delta}}{2}\Re(\overline{\psi^\e}D^\e D^\e_tD^\e_t \psi^\e)\\
		&=-\frac{\e^{2+2\delta}}{4}\pa_t\pa_t\nabla|\psi^\e|^2 + \e^{2+2\delta}\pa_t\Re(\overline{D^\e_t\psi^\e} D^\e \psi^\e)-\frac{\e^{2+2\delta}}{2} \Re(\overline{D_t^\e D_t^\e\psi^\e} D^\e \psi^\e )\\
		&\quad+ \e^{1+2\delta}\Im(\overline{\psi^\e}D_t^{\e}\psi^\e)(\e^{\delta}\pa_t A^\e-\nabla A^\e_0) ,\\
			I_3&:=\frac{\e^2}{2} \Re\left(\overline{\psi^\e}( D^\e D^\e_1D^\e_1\psi^\e+D^\e D^\e_2D^\e_2\psi^\e)\right)\\
			&=\frac{\e^2}{2}\Delta\Re\left(\overline{\psi^\e} D^\e \psi^\e \right) -\frac{\e^2}{2} \nabla\cdot \left(D^\e\psi^\e\otimes \overline{D^\e\psi^\e}+\overline{D^\e\psi^\e}\otimes D^\e\psi^\e\right)\\
			&\quad	+\frac{\e^2}{2}\Re(\overline{D^\e_1D^\e_1\psi^\e} D^\e \psi^\e+
			\overline{D^\e_2D^\e_2\psi^\e} D^\e \psi^\e )
	-\e^{1+\delta}(\pa_1A_2^\e-\pa_2 A^\e_1)\Im(\overline{\psi^\e}D^\e\psi^\e)^{\perp},\\
						I_4&:= -\nabla \left(V'\left(|\psi^\e|^2\right)\right)|\psi^\e|^2-V'\left(|\psi^\e|^2\right)\Re(\overline{\psi^\e}D^\e\psi^\e) \\
			&=-(\gamma-1)\nabla V\left(|\psi^\e|^2\right)-V'\left(|\psi^\e|^2\right)\Re(\overline{\psi^\e}D^\e\psi^\e).
	\end{align*}
Focusing on the third term, which requires special attention (the others are straightforward, and the second term can be derived using the same method), we analyze it step by step as follows:
\begin{align*}
	\begin{aligned}
		I_3&=\frac{\e^2}{2} \Re\left(\overline{\psi^\e} \left(D^\e D^\e_1D^\e_1\psi^\e+D^\e D^\e_2D^\e_2\psi^\e\right)\right)\\
		&=\frac{\e^2}{2}\sum_{j=1}^2\Re(\overline{\psi^\e}D^\e_j D^\e_j D^\e)-\e^{1+\delta}\sum_{j=1}^2\Im(\overline{\psi^\e}D^\e_j\psi^\e)(\pa_jA^\e-\nabla A^\e_j)\\        
		&=\frac{\e^2}{2} \sum_{j=1}^{2} \left(\pa_j \Re(\overline{\psi^\e} D^{\e}_j D^\e \psi^\e)
		-\Re(\overline{D^\e_j\psi^\e}D^\e_j D^\e \psi^\e)\right)-\e^{1+\delta}(\pa_1A_2^\e-\pa_2 A^\e_1)\Im(\overline{\psi^\e}D^\e\psi^\e)^{\perp}\\  
		&=\frac{\e^2}{2} \sum_{j=1}^{2} \left(\pa_j \pa_j \Re(\overline{\psi^\e} D^\e \psi^\e)
		-2\pa_j\Re(\overline{D^\e_j\psi^\e} D^\e \psi^\e )+\Re(\overline{D^\e_jD^\e_j\psi^\e} D^\e \psi^\e)\right)\\
		&\quad-\e^{1+\delta}(\pa_1A_2^\e-\pa_2 A^\e_1)\Im(\overline{\psi^\e}D^\e\psi^\e)^{\perp}\\
		&=\frac{\e^2}{4}\Delta\nabla|\psi^\e|^2-\frac{\e^2}{2} \nabla\cdot \left(D^\e\psi^\e\otimes \overline{D^\e\psi^\e}+\overline{D^\e\psi^\e}\otimes D^\e\psi^\e\right)\\
		&\quad+\frac{\e^2}{4}\Re\left((\overline{D^\e_1D^\e_1\psi^\e} +
		\overline{D^\e_2D^\e_2\psi^\e}) D^\e \psi^\e \right)-\e^{1+\delta}(\pa_1A_2^\e-\pa_2 A^\e_1)\Im(\overline{\psi^\e}D^\e\psi^\e)^{\perp}.
	\end{aligned}    
\end{align*}

Now, summing $ I_{\ell}$ for $ \ell=1,\dots,4 $, some terms cancel out due to $\eqref{A.2.1}$, and $\eqref{A.2}$ simplifies to:
\begin{align*}
	&-\e\pa_t \Im\left(\overline{\psi^\e} D^\e \psi^\e\right)-|\psi^\e|^2(\e^{\delta}\pa_t A^\e-\nabla A_0^\e)-\frac{\e^{2+2\delta}}{4}\pa_t\pa_t\nabla|\psi^\e|^2 + \e^{2+2\delta}\pa_t\Re\left(\overline{D^\e_t\psi^\e} D^\e \psi^\e\right) \\
	& + \e^{1+2\delta}\Im(\overline{\psi^\e}D_t^{\e}\psi^\e)(\e^{\delta}\pa_t A^\e-\nabla A^\e_0)  + \frac{\e^2}{4}\Delta\nabla|\psi^\e|^2 -
	\frac{\e^2}{2} \nabla\cdot \left(D^\e\psi^\e\otimes \overline{D^\e\psi^\e}+\overline{D^\e\psi^\e}\otimes D^\e\psi^\e\right)\\
	&-\e^{1+\delta}(\pa_1A_2^\e-\pa_2 A^\e_1)\Im(\overline{\psi^\e}D^\e\psi^\e)^{\perp}
	-(\gamma-1)\nabla V\left(|\psi^\e|^2\right) = 0.
\end{align*}
From $\eqref{A.1}_{2,3}$, the terms involving gauge fields cancel out:
\begin{align*}
	\big(-|\psi^\e|^2+\e^{1+2\delta}\Im(\overline{\psi^\e}D_t^{\e}\psi^\e)\big)(\e^{\delta}\pa_t A^\e-\nabla A_0^\e)  -\e^{\delta}(\pa_1A_2^\e-\pa_2 A^\e_1)\e\Im(\overline{\psi^\e}D^\e\psi^\e)^{\perp} = 0.
\end{align*}
Thus, we arrive at the conservation law for momentum:
\begin{align*}
	&\e\pa_t \Im\left(\overline{\psi^\e} D^\e \psi^\e\right)- \e^{2+2\delta}\pa_t\Re\left(\overline{D^\e_t\psi^\e} D^\e \psi^\e\right) +\frac{\e^2}{2} \nabla \cdot\left(D^\e\psi^\e\otimes \overline{D^\e\psi^\e}+\overline{D^\e\psi^\e}\otimes D^\e\psi^\e\right)\\
	&\quad -\frac{\e^2}{4}\Delta\nabla|\psi^\e|^2
	+(\gamma-1)\nabla V\left(|\psi^\e|^2\right)+\frac{\e^{2+2\delta}}{4}\pa_t\pa_t\nabla|\psi^\e|^2 =0.
\end{align*}

\vspace{0.2cm}
\noindent $\bullet$ (Total energy conservation): Multiplying $\eqref{A.1}_1$ by $ \overline{D^\e_t\psi^\e} $ and taking the real part, we obtain
\begin{align}\label{est-energy}
	\frac{\e^{2+2\delta}}{2}\Re\left(D_t^\e D_t^\e\psi^\e \overline{D_t^\e\psi^\e}\right) - \frac{\e^2}{2}\Re\left((D^\e_1D^\e_1\psi^\e+D^\e_2D^\e_2\psi^\e) \overline{D^\e_t\psi^\e}\right) + V'\left(|\psi^\e|^2\right)\Re(\psi^\e\overline{D^\e_t\psi^\e})=0.
\end{align}
Using previously derived identities, we rewrite each term as: for $j=1,2$,
\begin{align*}
	\Re(D_t^\e D_t^\e \psi^\e\overline{D_t^\e\psi^\e}) &= \frac{1}{2}\partial_t|D_t^\e\psi^\e|^2,\quad
	D_j^\e D_j^\e \psi^\e\overline{D^\e_t\psi^\e} = \pa_j\left(D^\e_j\psi^\e\overline{D^\e_t\psi^\e}\right) - \overline{D^\e_j D^\e_t\psi^\e} D^\e_j \psi^\e.
\end{align*}
The term $\overline{D^\e_j D^\e_t\psi^\e} D^\e_j \psi^\e$ was previously handled in the momentum conservation calculation, so applying the same argument, we obtain
\begin{align*}
	\Re(\overline{D^\e_j D^\e_t\psi^\e}D^\e_j \psi^\e) &= \frac{1}{2}\pa_t |D^{\e}_j \psi^\e|^2 - \e^{-1} (\e^{\delta}\pa_t A^\e_j - \pa_jA^\e_0)\Im\left(\overline{\psi^\e}D_j^\e\psi^\e\right).
\end{align*}
Substituting these into \eqref{est-energy}, we obtain
\begin{align*}
	\frac{\e^{2+2\delta}}{4}\pa_t|D^\e_t \psi^\e|^2
	+ \frac12\pa_t V\left(|\psi^\e|^2\right)
	+ \frac{\e^2}{4}\pa_t|D^\e\psi^\e|^2
	&= \frac{\e^2}{2}\nabla \cdot \Re\left(\overline{ D^\e_t\psi^\e}D^\e\psi^\e\right)\\
	&\quad + \frac{\e^2}{2}\Im(\overline{\psi^\e}(D^\e\psi^\e)^\perp)\cdot
	\Im(\overline{\psi^\e}D^\e\psi^\e).
\end{align*}
Since $v^\perp\cdot v=0$ for any $v\in\mathbb{R}^2$, the last term vanishes.
Hence we obtain a local energy balance, and integrating over $\mathbb{R}^2$
yields the conservation of total energy.
\end{proof}


\begin{thebibliography}{10}
		

 \bibitem{AC07} T. Alazard and R. Carles, \textit{Semi-classical limit of Schr\"odinger--Poisson equations in space dimension $n\ge3$}, J. Differential Equations {\bf 233} (2007), 241--275.
 
  \bibitem{BMP01} P. Bechouche, N. J. Mauser, and F. Poupaud, \textit{Semiclassical limit for the Schr\"odinger--Poisson equation in a crystal}, Comm. Pure Appl. Math. {\bf 54} (2001), 852--890.

\bibitem{BGL00} C. Bardos, F. Golse, and C. D. Levermore, \textit{The acoustic limit for the Boltzmann equation}, Arch. Rational Mech. Anal. {\bf 153} (2000), 177--204.

\bibitem{BMS04} P. Bechouche, N. J. Mauser, and S. Selberg, \textit{Nonrelativistic limit of Klein--Gordon--Maxwell to Schr\"odinger--Poisson}, Amer. J. Math. {\bf 126} (2004), 31--64.



 \bibitem{B09} N. Bournaveas, \textit{Low regularity solutions of the relativistic Chern--Simons--Higgs theory in the Lorenz gauge}, Electron. J. Differential Equations (2009), 1--10.

\bibitem{B00} Y. Brenier, \textit{Convergence of the Vlasov--Poisson system to the incompressible Euler equations}, Comm. Partial Differential Equations {\bf 25} (2000), 737--754.

\bibitem{CC02} D. Chae and K. Choe, \textit{Global existence in the Cauchy problem of the relativistic Chern--Simons--Higgs theory}, Nonlinearity {\bf 15} (2002), 747--758.

\bibitem{CH09} M. Chae and H. Huh, \textit{Semi-nonrelativistic limit of the Chern--Simons--Higgs system}, J. Math. Phys. {\bf 50} (2009), 072303.

 \bibitem{CY95} L. A. Caffarelli and Y. Yang, \textit{Vortex condensation in the Chern--Simons--Higgs model: An existence theorem}, Commun. Math. Phys. {\bf 168} (1995), 321--336.

\bibitem{CJ21} Y.-P. Choi and J. Jung, \textit{Asymptotic analysis for a Vlasov--Fokker--Planck/Navier--Stokes system in a bounded 
	domain}, Math. Models Methods Appl. Sci. {\bf 31} (2021), 2213--2295.
	
\bibitem{D95} G. V. Dunne, \textit{Self-dual Chern–Simons theories}, Springer, 1995.

 \bibitem{EHI92} Z. F. Ezawa, M. Hotta, and A. Iwasaki, \textit{Anyon field theory and fractional quantum Hall statics}, Prog. Theor. Phys. Suppl. {\bf 107} (1992), 185--194.

\bibitem{HS07} J. Han and K. Song, \textit{Nonrelativistic limit in the self-dual abelian Chern--Simons model}, J. Korean Math. Soc.
{\bf 44} (2007), 997--1012.

\bibitem{HKP90} J. Hong, Y. Kim, and P. Y. Pac, \textit{Multivortex solutions of the Abelian Chern--Simons--Higgs theory}, Phys. Rev. Lett. {\bf 64} (1990), 2230--2233.

\bibitem{H07} H. Huh, \textit{Local and global solutions of the Chern--Simons--Higgs system}, J. Funct. Anal. {\bf 242} (2007), 526--549.

\bibitem{H11} H. Huh, \textit{Towards the Chern--Simons--Higgs equation with finite energy}, Discrete Contin. Dyn. Syst. {\bf 30} (2011), 1145--1159.

\bibitem{HM26} H. Huh and B. Moon, \textit{Understanding the non-relativistic behavior of the Chern--Simons--Higgs system}, J. Nonlinear Math. Phys., {\bf 33} (2026), Article No. 22.

\bibitem{HO16} H. Huh and T. Oh, \textit{Low regularity solutions to the Chern--Simons--Dirac and the Chern--Simons--Higgs equations in the Lorenz gauge}, Commun. PDE {\bf 41} (2016), 375--397.


\bibitem{JP90} R. Jackiw and S.-Y. Pi, \textit{Classical and quantal nonrelativistic Chern--Simons theory}, Phys. Rev. D {\bf 42} (1990), 3500--3513.



\bibitem{JW90} R. Jackiw and E. J. Weinberg, \textit{Self-dual Chern--Simons vortices}, Phys. Rev. Lett. {\bf 64} (1990), 2234--2237.

 
\bibitem{JMS11} S. Jin, P. Markowich, and C. Sparber, \textit{Mathematical and computational methods for semiclassical Schr\"odinger equations}, Acta Numer. {\bf 20} (2011), 121--209.

 \bibitem{JS21} S. Jin and J. Seok, \textit{Nonrelativistic limit of solitary waves for nonlinear Maxwell--Klein--Gordon equations}, Calc. Var. Partial Differential Equations {\bf 60} (2021), Paper No. 168.



 \bibitem{JW03} A. J\"ungel and S. Wang, \textit{Convergence of nonlinear Schr\"odinger--Poisson systems to the compressible Euler equations}, Comm. Partial Differential Equations {\bf 28} (2003), 1005--1022.
 


\bibitem{KM22} J. Kim and B. Moon, \textit{Hydrodynamic limits of the nonlinear Schr\"odinger equation with the Chern--Simons gauge fields}, Discrete Contin. Dyn. Syst. {\bf 42} (2022), 2541--2561.

\bibitem{KM23a} J. Kim and B. Moon, \textit{Hydrodynamic limits of Manton's Schr\"odinger system}, Commun. Pure Appl. Anal. {\bf 22} (2023), 2278--2297.

\bibitem{KM24}  J. Kim and B. Moon, \textsl{Hydrodynamic limit of the Maxwell--Schrödinger equations to the compressible Euler--Maxwell equations}, J. Differential Equations, {\bf 397} (2024), 34--54. 


 \bibitem{KM25} J. Kim and B. Moon, \textsl{Quantified asymptotic analysis for the relativistic quantum mechanical system with electromagnetic fields}, J. Math. Anal. Appl., {\bf 543} (2025), Paper No. 125800, 28 pp.	

\bibitem{KM23b} J. Kim and B. Moon, \textit{Quantified hydrodynamic limits for Schr\"odinger-type equations without the nonlinear potential}, J. Evol. Equ. {\bf 23} (2023), Paper No. 51, 27 pp.

\bibitem{LL03} H. Li and C.-K. Lin, \textit{Semiclassical limit and well-posedness of nonlinear Schr\"odinger--Poisson systems}, Electron. J. Differ. Equ. {\bf 2003} (2003), 1--17.

\bibitem{L98} P.-L. Lions, \textit{Mathematical topics in fluid mechanics}, Oxford University Press, 1998.

\bibitem{LM98} P.-L. Lions and N. Masmoudi, \textit{Incompressible limit for a viscous compressible fluid}, J. Math. Pures Appl. {\bf 77} (1998), 585--627.

\bibitem{LW12} C.-K. Lin and K.-C. Wu, \textit{Hydrodynamic limits of the nonlinear Klein--Gordon equation}, J. Math. Pures Appl. {\bf 98} (2012), 328--345.

\bibitem{M84} A. Majda, \textit{Compressible fluid flow and systems of conservation laws in several space variables}, Springer-Verlag, 1984.

\bibitem{M27} E. Madelung, \textit{Quantentheorie in hydrodynamischer Form}, Z. Physik {\bf 40} (1927), 322--326.

\bibitem{MN03} N. Masmoudi and K. Nakanishi, \textit{Nonrelativistic limit from Maxwell--Klein--Gordon and Maxwell--Dirac to Poisson--Schr\"odinger}, Int. Math. Res. Not. {\bf 2003} (2003), 697--734.

\bibitem{MV08} A. Mellet and A. Vasseur, \textit{Asymptotic analysis for a Vlasov--Fokker--Planck/compressible Navier--Stokes system}, Commun. Math. Phys. {\bf 281} (2008), 573--596.


 \bibitem{P02} M. Puel, \textit{Convergence of the Schr\"odinger--Poisson system to the incompressible Euler equations}, Comm. Partial Differential Equations {\bf 27} (2002), 2311--2331.


\bibitem{S09} L. Saint-Raymond, \textit{Hydrodynamic limits: some improvements of the relative entropy methods}, Ann. Inst. H. Poincar\'e Anal. Non Lin\'eaire {\bf 26} (2009), 705--744.


\bibitem{S79} A. Y. Schoene, \textit{On the nonrelativistic limits of the Klein--Gordon and Dirac equations}, J. Math. Anal. Appl. {\bf 71} (1979), 36--47.


\bibitem{ST13} S. Selberg and A. Tesfahun, \textit{Global well-posedness of the Chern--Simons--Higgs equations with finite energy}, Discrete Contin. Dyn. Syst. {\bf 33} (2013), 2531--2546.



 \bibitem{Y01} Y. Yang, \textit{Solitons in Field Theory and Nonlinear Analysis}, Springer Monographs in Mathematics, Springer, 2001.


 \bibitem{Y11} J. Yuan, \textit{Local well-posedness of Chern--Simons--Higgs system in the Lorenz gauge}, J. Math. Phys. {\bf 52} (2011), 103706.
 

 



		
		
	\end{thebibliography}
\end{document}